\documentclass{jms_My_en}
\paperTitle{The calculation of the probability density and distribution function of a strictly stable law in the vicinity of zero }
\articleColonName{The calculation of a strictly stable law in the vicinity of zero}
\authorsShort{V.\,V.\,Saenko}
\authorsFull{V.\,V. Saenko\first}
\addAuthorInfo{Ulyanovsk State University, S.P. Kapitsa Research Institute of Technology, city of  Ulyanovsk, 42, Leo Tolstoy St.,  432017, e-mail: \url{vvsaenko@inbox.ru}}

\paperAbstract{The problem of calculating the probability density and distribution function of a strictly stable law is considered at $x\to0$. The expansions of these values into power series were obtained to solve this problem. It was shown that in the case $\alpha<1$ the obtained series were asymptotic at $x\to0$, in the case  $\alpha>1$ they were convergent and in the case $\alpha=1$ in the domain  $|x|<1$ these series converged to an asymmetric Cauchy distribution. It has been shown that at $x\to0$ the obtained expansions can be successfully used to calculate the probability density and distribution function of strictly stable laws.}
\usepackage{color}
\begin{document}
\maketitle

\section{Introduction}

The major inconvenience of using stable laws is the absence of expressions for probability density and distribution function in terms of elementary functions. There are only five cases known when the density is expressed in terms of elementary functions: the L$\acute{e}$vy distribution ($\alpha=1/2,\theta=1$)  symmetric L$\acute{e}$vy distribution ($\alpha=1/2,\theta=-1$), Cauchy distribution ($\alpha=1,\theta=0$), The Gaussian distribution  ($\alpha=2,\theta=0$) and asymmetric Cauchy distribution ($\alpha=1,-1\leqslant\theta\leqslant1$) (see formulas (\ref{eq:gx_a1}) and (\ref{eq:Gx_a1})). Here $\alpha$ is characteristic exponent of a stable law, $\theta$ - is  a parameter of asymmetry. The latter distribution first came out in the book by V.M. Zolotarev \cite{Zolotarev1986} (see formula (2.3.5a)) and later was examined in the works \cite{Saenko2020c,Saenko2020b}.   Different representations for stable laws are required to calculate the probability density or distribution function in other cases.

The paper \cite{Zolotarev1995_en}  shows that if  values  of the characteristic exponent  $\alpha$ and the asymmetry parameter $\beta$ are limited by values of  rational numbers ($\alpha=P/Q$, $\beta=U/V$, where $P,Q,V$ are positive integers), then in this case it is possible to express the probability density of a strictly stable law in terms of special functions. The papers \cite{Schneider1986, Schneider1987,HoffmannJorgensen1994,Zolotarev1995_en,Penson2010,Gorska2011,Pogany2015} are  devoted to obtaining such representations. The limitation of this approach lies in the fact that it is possible to obtain an expression for the probability density only for rational values of the parameters  $\alpha$ and $\beta$, and only for strictly stable laws. The application of the Fast Fourier Transform algorithm is another method of calculating density.  This approach has been examined in the papers \cite{Mittnik1999,Menn2006}.   However, this method gives an opportunity to calculate the probability density on a grid of equidistant points. In the paper linear interpolation must be used to calculate the density at intermediate points or at irregularly spaced points.

The use of integral representations is the main method for calculating the probability density and the distribution function of stable laws. This approach is based on the inversion formula (\ref{eq:InverseFormula}).  There are two possible ways of inverting the characteristic function. The first way is to directly calculate the integral in (\ref{eq:InverseFormula}). As a result, the probability density is expressed in terms of the integral of the oscillating function \cite{Nolan1999,Ament2018}. However, since the integrand is an oscillating function, this leads to difficulties in numerical integration in the cases $\alpha<0.75$, $\beta\neq0$ and $0<|\alpha-1|<0.001$ and in the case of large values $x$ \cite{Nolan1999}. Modernization of the standard quadrature method of numerical integration makes it possible to reduce the lower boundary of the parameter $\alpha$ from the value $0.75$ to the value  $0.5$ \cite{Ament2018}.  It is proposed to use the representation of the density in the form of a power series to calculate the density for large values of $x$.

The second way of obtaining integral representations is the application of the stationary phase method when calculating the integral in (\ref{eq:InverseFormula}) (see \cite{Zolotarev1964_en,Zolotarev1986,Nolan1997,Saenko2020b}). The advantage of this method of inverting the characteristic function is that the resulting integral representation is expressed in terms of a definite integral of a monotonic function. Such integral representations were obtained for stable laws with different parameterizations of the characteristic function: for parameterization ``B" in the works   \cite{Zolotarev1964_en,Zolotarev1986}, for parameterization  ``M" in the paper \cite{Nolan1997}, for parameterization ``C" in the paper \cite{Saenko2020b}. (Here, the notation of various parameterizations of the characteristic function is given in accordance with the designations introduced in the book by V.M. Zolotarev \cite{Zolotarev1986}.) These integral representations are more convenient from a practical point of view and allow calculating the density in a wide range of parameter values $\alpha,\beta$ and coordinates $x$. The integral representation obtained in the work \cite{Nolan1997} served as a foundation for developing several software products \cite{Liang2013,Royuela-del-Val2017,Julian-Moreno2017,Rimmer2005,Veillette2008}.

From a theoretical point of view, these integral representations are valid for all values of $x$. However, in practice, it is not possible to calculate the probability density and distribution function for all values of $x$. The reason for this lies in the behavior of the integrand.  The integrand has the form of a very sharp peak with small and large values of $x$. As a result, numerical integration algorithms cannot correctly calculate the integral in this range of $x$. To settle this issue in the papers \cite{Royuela-del-Val2017,Julian-Moreno2017,Nolan1997} it is proposed to use various numerical methods to increase the accuracy of calculations. However, all proposed approaches increase the accuracy of the calculation, but do not completely eliminate the problem.  To calculate the probability density and distribution function in this range of values of $x$ it is expedient to use other representations for stable laws which do not have any specific features in the indicated areas. The approach used in the papers  \cite{Ament2018,Menn2006} seems to be the most suitable which consists in applying expansions in a power series for probability density and distribution function with $x\to0$ and $x\to\infty$.

Such expansions are well known and are obtained, as a rule, for parametrization ``B". Depending on the value of the parameter $\alpha$ the obtained power series is either convergent or asymptotic.  The expansion of the probability density of a stable law into a convergent series in the case $x\to\infty$ and $0<\alpha<1$, was firstly mentioned in the paper \cite{Pollard1946}. Later, in the paper\cite{Bergstrom1952}  a generalization of this density expansion was given for $x\to\infty$ in the case $1<\alpha<2$. In this range of values of the parameter $\alpha$ this series turns out to be asymptotic. In the same paper, the expansion of the density in a series in the vicinity of the point  $x\to0$ was obtained for the case $0<\alpha<2$.  The resulting power series is asymptotic in the case $0<\alpha<1$, and convergent in the case $1<\alpha<2$. Expansions for $\alpha>1$ in the cases $x\to0$ and $x\to\infty$ were also obtained in the work \cite{Schneider1986} as a result of expansion into a power series of the probability density, expressed in terms of the Fox function.  The same expansions were given in the books \cite{Feller1971_V2_en} (see Chapter 17, \S7) and \cite{Zolotarev1986} (see \S2.4 and \S2.5).  Expansions of the density of a stable law in a power series for the characteristic function in parameterization ``M" were obtained in the paper \cite{Ament2018}. An interesting result was obtained in the paper \cite{Arias-Calluari2017}. In this paper, expansions in power series were obtained for the probability density of a symmetric stable law at $x\to0$ and $x\to\infty$ for the cases $0<\alpha<1$ and $1<\alpha<2$. A distinctive property of this expansion is that these power series for all $\alpha$ are convergent.

The purpose of this work is to obtain power series expansions of the probability density and distribution function of a strictly stable law with the characteristic function
\begin{equation}\label{eq:CF_formC}
  \hat{g}(t,\alpha,\theta,\lambda)=\exp\left\{-\lambda |t|^\alpha\exp\{-i\tfrac{\pi}{2}\alpha\theta\sign t\}\right\},\quad t\in\mathbf{R},
\end{equation}
where $\alpha\in(0,2]$, $|\theta|\leqslant\min(1,2/\alpha-1)$, $\lambda>0$. This parameterization of the characteristic function, according to the book \cite{Zolotarev1986}, is called parameterization ``C". Obtaining such expansions turns out to be necessary in connection with the problem of calculating the probability density and distribution function of stable and fractionally stable laws. In fact, in the article \cite{Saenko2020b} integral representations were obtained for the probability density and distribution function of a strictly stable law with the characteristic function (\ref{eq:CF_formC}). Since these integral representations were obtained using the stationary phase method, then with small and large values of the coordinate $x$ the integrand has the form of a very sharp peak.   This causes difficulties for numerical integration algorithms and leads to incorrect integration results. Therefore, to calculate the probability density and distribution function in these coordinate regions, it is expedient to use representations in the form of a power series for the corresponding quantities. This work is devoted to obtaining such expansions.

The solution to this problem will turn out to be useful not only when calculating the density of strictly stable laws but also in the task of calculation the density and distribution function of a fractional-stable law \cite{Kolokoltsov2001,Bening2006,Saenko2020c}. These distributions are expressed in terms of the Mellin convolution of two strictly stable laws. Correct calculation of the probability density will make it possible to use an algorithm for statistical estimation of the parameters of these laws based on the maximum likelihood method. Such an algorithm for estimating parameters will give an opportunity to correctly describe various experimental data. It is known that the distribution of gene expression is described by laws with a power-law decrease in density \cite{Ueda2004a,Hoyle2002,Furusawa2003}. Since the stable and fractionally stable densities decrease according to the power law $x^{-\alpha-1}$ at $x\to \infty$, then these classes of distributions were used to describe the distribution of gene expression. In the works \cite{Saenko2015,Saenko2015a} fractional stable distributions were used to describe the expression of genes obtained using microarray technology. In the work \cite{Saenko2016a} these distributions were used to describe the results obtained using the Next Generation Sequence technology. To describe these experimental data, it is necessary to have algorithms for statistical estimation of parameters, the most effective of which is the maximum likelihood method. To construct such an algorithm, it is necessary to be able to correctly calculate the density of a strictly stable law for any values of $x$.

\section{Preliminary remarks}

The major purpose is to obtain the expansion of the density and distribution function of a strictly stable law in a power series in the vicinity of the point $x=0$. The paper deals with strictly stable laws with the characteristic function (\ref{eq:CF_formC}). Without loss of generality, we will assume that the scale parameter $\lambda=1$. Strictly stable laws with the parameter $\lambda=1$ are commonly called standard strictly stable laws.  Designation abbreviations are accepted for standard strictly stable laws.  The characteristic function will be designated by $\tilde{g}(t,\alpha,\theta,1)\equiv \hat{g}(t,\alpha,\theta)$, the probability density distribution will be designated by  $g(x,\alpha,\theta,1)\equiv g(x,\alpha,\theta)$, the distribution function will be designated by  $G(x,\alpha,\theta,1)\equiv G(x,\alpha,\theta)$.

To perform the inverse Fourier transform and obtain the probability density distribution, the following lemma is useful, which defines the inversion formula

\begin{lemma}
The probability density function $g(x,\alpha,\theta)$ for any admissible set of parameters $(\alpha,\theta)$ and any $x$ can be obtained using the inversion formulas
\begin{equation}\label{eq:InverseFormula}
  g(x,\alpha,\theta)=\frac{1}{2\pi}\int_{-\infty}^{\infty}e^{-itx}\hat{g}(t,\alpha,\theta)dt=
  \left\{\begin{array}{c}
           \displaystyle\frac{1}{\pi}\Re\int_{0}^{\infty} e^{itx}\hat{g}(t,\alpha,-\theta)dt,  \\
           \displaystyle\frac{1}{\pi}\Re\int_{0}^{\infty} e^{-itx}\hat{g}(t,\alpha,\theta)dt.
         \end{array}\right.
\end{equation}
\end{lemma}

The proof of this lemma can be found in the paper \cite{Saenko2020b}. To obtain the probability density, there is no fundamental difference which of the formulas to use on the right side (\ref{eq:InverseFormula}). The result will differ only in the sign of the parameter $\theta$. Without loss of generality, in this paper we will use the first formula (\ref{eq:InverseFormula}). Such a choice results from the fact that in the works \cite{Zolotarev1986,Uchaikin1999,Saenko2020b} this formula was used to invert the characteristic function. This will give us an opportunity to compare the results obtained below with the results of the mentioned papers without any additional transformations.

In the article \cite{Saenko2020b} the inverse Fourier transform of the characteristic function  (\ref{eq:CF_formC}) was performed and expressions for the probability density and distribution function of a strictly stable law were obtained.  In the case $\alpha\neq1$ and $x\neq0$ for any admissible  $\theta$ the following integral representation is true for the probability density
\begin{equation}\label{eq:pdfI}
  g(x,\alpha,\theta)=\frac{\alpha}{\pi|\alpha-1|}\int_{-\pi\theta^*/2}^{\pi/2}\exp\left\{-|x|^{\alpha/(\alpha-1)}U(\varphi,\alpha,\theta^*)\right\} U(\varphi,\alpha,\theta^*)|x|^{1/(\alpha-1)}d\varphi,
\end{equation}
where $\theta^*=\theta\sign(x)$ and
\begin{equation}\label{eq:U}
  U(\varphi,\alpha,\theta)=\left(\frac{\sin\left(\alpha\left(\varphi+\frac{\pi}{2}\theta\right)\right)} {\cos\varphi}\right)^{\frac{\alpha}{1-\alpha}}\frac{\cos\left(\varphi(1-\alpha)-\frac{\pi}{2}\alpha\theta\right)}{\cos\varphi}.
\end{equation}
If $\alpha=1$, then for any admissible $-1\leqslant\theta\leqslant1$ the probability density has the form
\begin{equation}\label{eq:gx_a1}
  g(x,1,\theta)=\frac{\cos(\pi\theta/2)}{\pi\left(x^2-2x\sin(\pi\theta/2)+1\right)}.
\end{equation}
If $x=0$, then $g(0,\alpha,\theta)=\frac{1}{\pi}\cos(\pi\theta/2)\Gamma(1/\alpha+1)$.

The following expressions are valid for the distribution function. If $\alpha\neq1$, then for any admissible $\theta$
\begin{equation}\label{eq:cdfI}
  G(x,\alpha,\theta)=\tfrac{1}{2}(1-\sign(x))+\sign(x)G^{(+)}(|x|,\alpha,\theta^*),
\end{equation}
where
\begin{equation}\label{eq:cdfI_G+}
  G^{(+)}(x,\alpha,\theta)=1-\frac{1+\theta}{4}(1+\sign(1-\alpha))+
  \frac{\sign(1-\alpha)}{\pi}\int_{-\pi\theta/2}^{\pi/2}\exp\left\{-x^{\alpha/(\alpha-1)}U(\varphi,\alpha,\theta)\right\}d\varphi,
\end{equation}
$x>0$ and $U(\varphi,\alpha,\theta)$ is determined by the expression (\ref{eq:U}). If $\alpha=1$, then for any  $-1\leqslant\theta\leqslant1$
\begin{equation}\label{eq:Gx_a1}
G(x,1,\theta)=\frac{1}{2}+\frac{1}{\pi}\arctan\left(\frac{x-\sin\left(\frac{\pi}{2}\theta\right)} {\cos\left(\frac{\pi}{2}\theta\right)}\right).
\end{equation}
In the point $x=0$ for any admissible $\alpha$ and $\theta$
\begin{equation}\label{eq:G_x=0}
G(0,\alpha,\theta)=\frac{1}{2}(1-\theta).
\end{equation}

To obtain the density representation $g(x,\alpha,\theta)$ in the form of power series the integral obtained in the book  \cite{Bateman_V1_1953} (see \S 1.5. formula (31)) turns out to be useful.
\begin{equation*}
  \int_{0}^{\infty}t^{\gamma-1}e^{-ct\cos\beta-ict\sin\beta}dt=\Gamma(\gamma)c^{-\gamma}e^{-i\gamma\beta},\
  -\frac{\pi}{2}<\beta<\frac{\pi}{2},\ \Re\gamma>0\ \text{or}\ \beta=\pm\frac{\pi}{2},\ 0<\Re\gamma<1.
\end{equation*}
If we use Euler's formula $\cos\beta+i\sin\beta=e^{i\beta}$, then this integral can be represented in the form
\begin{equation}\label{eq:Gamma_intRepr1}
  \int_{0}^{\infty}t^{\gamma-1}e^{-ct\exp\{i\beta\}}dt=\Gamma(\gamma)c^{-\gamma}e^{-i\gamma\beta},\quad
  -\frac{\pi}{2}<\beta<\frac{\pi}{2},\ \Re\gamma>0\ \text{or}\ \beta=\pm\frac{\pi}{2},\ 0<\Re\gamma<1.
\end{equation}

\section{Representation of the probability density in the form of  a power series}

We obtain the expansion of the probability density $g(x,\alpha,\theta)$  in a series at $x\to0$.  The following theorem is valid

\begin{theorem}\label{th:gx0_expan}
In the case $x\to0$  for any admissible set of parameters $(\alpha,\theta)$ except for the values $\alpha=1, \theta=\pm1$ for the probability density $g(x,\alpha,\theta)$ the following representation in the form of a series is valid
\begin{equation}\label{eq:gx0}
g(x,\alpha,\theta)= g_N^0(x,\alpha,\theta)+R_N^0,
\end{equation}
where
\begin{align}
   g_N^0(x,\alpha,\theta)&= \frac{1}{\alpha\pi}\sum_{n=0}^{N-1}\frac{x^n}{n!}\Gamma\left(\frac{n+1}{\alpha}\right) \sin\left(\tfrac{\pi}{2}(n+1)(1-\theta)\right),\label{eq:gN0}\\
  |R_N^0|&\leqslant\frac{1}{\alpha\pi}\frac{|x|^N}{N!}\Gamma\left(\frac{N+1}{\alpha}\right)\label{eq:absRN}
\end{align}
\end{theorem}

\begin{proof}
We will perform the inverse Fourier transform of the characteristic function (\ref{eq:CF_formC}). To do this we make use of the first relation in (\ref{eq:InverseFormula}). We have
\begin{equation*}
g(x,\alpha,\theta)=\frac{1}{\pi}\Re\int_{0}^{\infty}e^{itx}\hat{g}(t,\alpha,-\theta)dt=
\frac{1}{\pi}\Re\int_{0}^{\infty}\exp\left\{itx-t^\alpha\exp\left\{i\tfrac{\pi}{2}\alpha\theta\right\}\right\}dt.
\end{equation*}

Since the considered case $x\to0$, then we expand $\exp\{itx\}$ in a series in the vicinity of the point  $x=0$. As a result, we get
\begin{equation}\label{eq:g_expan}
  g(x,\alpha,\theta)=g_N^0(x,\alpha,\theta)+R_N^0(x,\alpha,\theta),
\end{equation}
where the $N$-th partial sum $g_N^0(x,\alpha,\theta)$ and the remainder $R_N^0(x,\alpha,\theta)$ of a series have the form
\begin{align}
  g_N^0(x,\alpha,\theta)&=\frac{1}{\pi}\Re\int_{0}^{\infty}\exp\left\{-t^\alpha\exp\left\{i\tfrac{\pi}{2}\alpha\theta\right\}\right\}
  \sum_{n=0}^{N-1}\frac{(itx)^n}{n!}dt,\label{eq:gx0_tmp1}\\
 R_N^0(x,\alpha,\theta)&=\frac{1}{\pi}\Re\int_{0}^{\infty}\exp\left\{-t^\alpha\exp\left\{i\tfrac{\pi}{2}\alpha\theta\right\}\right\}
  R_N(itx)dt.\label{eq:RNx0_tmp1}
\end{align}
Here $R_N(itx)=\frac{(itx)^N}{N!}e^{itx\zeta}, (0<\zeta<1)$  is the remainder in the Lagrange form.

We consider the $N$-th partial sum $g_N^0(x,\alpha,\theta)$. To calculate the integral in (\ref{eq:gx0_tmp1}), we will change in some places the order of summation and integration and we will substitute the  integration variable $t^\alpha=\tau$. As a result, we obtain
\begin{equation}\label{eq:gx0_tmp2}
  g_N^0(x,\alpha,\theta) =
  \frac{1}{\alpha\pi}\sum_{n=0}^{N-1}\frac{x^n}{n!}\Re i^n\int_{0}^{\infty}\tau^{\frac{n+1}{\alpha}-1} \exp\left\{-\tau\exp\left\{i\tfrac{\pi}{2}\alpha\theta\right\}\right\}d\tau.
\end{equation}

Next, we examine the range of valid values of the argument $\tfrac{\pi}{2}\alpha\theta$. The range of admissible values of the parameter $\theta$ is determined by the inequality  $|\theta|\leqslant\min(1,2/\alpha-1)$. Hence, if $0<\alpha\leqslant1$, then $-1\leqslant\theta\leqslant1$, if $1<\alpha\leqslant2$, then $-(2/\alpha-1)\leqslant\theta\leqslant2/\alpha-1$. Thus,
\begin{equation}\label{eq:arg_a01}
  \tfrac{\pi}{2}\alpha\theta\in\left[\tfrac{\pi}{2}\alpha,\tfrac{\pi}{2}\alpha\right]\in\left[-\tfrac{\pi}{2},\tfrac{\pi}{2}\right],\quad
  \mbox{if}\quad 0<\alpha\leqslant1.
\end{equation}
\begin{equation}\label{eq:arg_a12}
  \tfrac{\pi}{2}\alpha\theta\in\left[\tfrac{\pi}{2}\alpha-\pi,\pi-\tfrac{\pi}{2}\alpha\right] \in \left(-\tfrac{\pi}{2},\tfrac{\pi}{2}\right),\quad
  \mbox{if}\quad 1<\alpha\leqslant2.
\end{equation}
Combining (\ref{eq:arg_a01}) and (\ref{eq:arg_a12}), we obtain
\begin{equation}\label{eq:arg_a02}
  -\tfrac{\pi}{2}\leqslant\tfrac{\pi}{2}\alpha\theta\leqslant\tfrac{\pi}{2},\quad \mbox{if}\quad 0<\alpha\leqslant2.
\end{equation}
As we can see, extreme values of this range are reached in the case $\alpha=1$ and $\theta=\pm1$.

Taking into consideration (\ref{eq:arg_a02}), it is clear that  to calculate the integral in  (\ref{eq:gx0_tmp2}), one can use the formula (\ref{eq:Gamma_intRepr1}). We get
\begin{equation}\label{eq:gx0_intGammaFun}
 \int_{0}^{\infty}\tau^{\frac{n+1}{\alpha}-1} \exp\left\{-\tau\exp\left\{i\tfrac{\pi}{2}\alpha\theta\right\}\right\}d\tau=
 \Gamma\left(\frac{n+1}{\alpha}\right)\exp\left\{-i\tfrac{\pi}{2}(n+1)\theta\right\}.
\end{equation}
From the relation (\ref{eq:Gamma_intRepr1}) it follows that for the arbitrary value $(n+1)/\alpha>0$ it is necessary to exclude the case $\tfrac{\pi}{2}\alpha\theta=\pm\tfrac{\pi}{2}$ from consideration, which is implemented at values  $\alpha=1,\theta=\pm1$. Now using the expression  (\ref{eq:gx0_intGammaFun}) in (\ref{eq:gx0_tmp2}), we obtain
\begin{equation*}
g_N^0(x,\alpha,\theta)=\frac{1}{\alpha\pi}\sum_{n=0}^{N-1}\frac{x^n}{n!}\Gamma\left(\frac{n+1}{\alpha}\right)\Re i^n \exp\left\{-i\tfrac{\pi}{2}(n+1)\theta\right\}.
\end{equation*}
Considering now that $\Re i^{n} \exp\left\{-i\tfrac{\pi}{2}(n+1)\theta\right\}=\sin\left(\tfrac{\pi}{2}(n+1)(1-\theta)\right)$, we finally obtain
\begin{equation*}
  g_N^0(x,\alpha,\theta)=\frac{1}{\alpha\pi}\sum_{n=0}^{N-1}\frac{x^n}{n!}\Gamma\left(\frac{n+1}{\alpha}\right) \sin\left(\tfrac{\pi}{2}(n+1)(1-\theta)\right).
\end{equation*}

Now we consider the remainder $R_N^0(x,\alpha,\theta)$.  From the expression (\ref{eq:RNx0_tmp1}) we get
\begin{equation*}
  R_N^0=\frac{x^N}{\pi N!}\Re i^N\int_{0}^{\infty}\exp\left\{itx\zeta-t^\alpha\exp\left\{i\tfrac{\pi}{2}\alpha\theta\right\}\right\}
  t^Ndt.
\end{equation*}
It is not possible to calculate this integral since the exact value of the quantity $\zeta$ is not known. It is only known that $0<\zeta<1$. However, one can obtain an estimate for this integral. We have
\begin{multline*}
  R_N^0(x,\alpha)\leqslant|R_N^0(x,\alpha)| \leqslant \frac{1}{\pi}\frac{|x|^N}{N!}\left|\Re i^N\int_{0}^{\infty}t^N \exp\left\{-t^\alpha\exp\left\{i\tfrac{\pi}{2}\alpha\theta\right\}\right\}e^{itx\zeta}dt\right|\leqslant \\
  \leqslant \frac{|x|^N}{\pi N!}\left(\left|\Re i^N\int_{0}^{\infty}t^N \exp\left\{-t^\alpha\exp\left\{i\tfrac{\pi}{2}\alpha\theta\right\}\right\}dt\right|\right)=\\
   = \frac{|x|^N}{\alpha\pi N!}\left(\left|\Re i^N\int_{0}^{\infty}\tau^{\frac{N+1}{\alpha}-1} \exp\left\{-\tau\exp\left\{i\tfrac{\pi}{2}\alpha\theta\right\}\right\}d\tau\right|\right)=\\
   \frac{|x|^N}{\alpha\pi N!} \Gamma\left(\frac{N+1}{\alpha}\right) \left|\Re i^N\exp\left\{-i\tfrac{\pi}{2}(N+1)\theta\right\}\right|\leqslant
   \frac{1}{\alpha\pi}\frac{|x|^N}{N!} \Gamma\left(\frac{N+1}{\alpha}\right)
\end{multline*}
To obtain the third inequality, it was taken into consideration that $|\exp\{itx\zeta\}|\leqslant1$. Next, the integration variable was substituted $t^\alpha=\tau$. To calculate the resulting integral, the formula (\ref{eq:Gamma_intRepr1}) was used.
The obtained expression completely proves the theorem.
\begin{flushright}
  $\Box$
\end{flushright}
\end{proof}

We need to make one small remark. When proving the theorem, it was pointed out that it was necessary to exclude the case $\tfrac{\pi}{2}\alpha\theta=\pm\tfrac{\pi}{2}$ from consideration, which corresponds to the values of parameters $\alpha=1, \theta\pm1$. As part of the proof of the theorem, this was done so that the range of admissible values of the argument $\tfrac{\pi}{2}\alpha\theta$  of the integral in (\ref{eq:gx0_tmp2}), should coincide with the range of admissible values of the argument $\beta$, included in the integral(\ref{eq:Gamma_intRepr1}). However, the exception of the case $\beta=\pm\tfrac{\pi}{2}$ from the integral (\ref{eq:Gamma_intRepr1}) is related with the fact that in these points the integral (\ref{eq:Gamma_intRepr1}) will diverge (for details see\cite{Saenko2022a}). Therefore, it should be assumed that $\alpha=1$ and $\theta=\pm1$, then the integral in (\ref{eq:gx0_tmp2}) will diverge. This in its turn leads to a degenerate probability density at that point. As a result, we arrive at the well-known fact that the probability density with the characteristic function  (\ref{eq:CF_formC}) is degenerate in the points $\alpha=1,\theta\pm1$.

As noted in the Introduction, depending on the value of the parameter $\alpha$, the expansion of the probability density of the stable law in a power series turns out to be either convergent or divergent.   Absolutely the same situation occurs in the considered case.   The answer to the question under what values of $\alpha$ the expansion (\ref{eq:gx0}) is convergent, and for which it is divergent we formulate as a corollary

\begin{corollary}\label{corol:gN0_a><1}
In the case $\alpha<1$ the series (\ref{eq:gN0}) is divergent at $N\to\infty$. In this case for the density $g(x,\alpha,\theta)$ for any admissible $\theta$ the asymptotic expansion
\begin{equation*}
  g(x,\alpha,\theta)\sim \frac{1}{\alpha\pi}\sum_{n=0}^{N-1}\frac{x^n}{n!}\Gamma\left(\frac{n+1}{\alpha}\right) \sin\left(\tfrac{\pi}{2}(n+1)(1-\theta)\right),\quad x\to0.
\end{equation*} is valid.

In the case $\alpha=1$ the series (\ref{eq:gN0}) converge for any $x$, satisfying the condition $|x|<1$. In this case for the density $g(x,1,\theta)$ for any admissible $\theta\neq\pm1$ it is possible to represent in the form of an infinite series
\begin{equation}\label{eq:gx0_a1}
  g(x,1,\theta)=\frac{1}{\pi}\sum_{n=0}^{\infty}\sin(\tfrac{\pi}{2}(n+1)(1-\theta)) x^n, \quad |x|<1.
\end{equation}

In the case $\alpha>1$ the series (\ref{eq:gN0}) converge for any $x$. In this case for the density $g(x,\alpha,\theta)$ for any admissible $\theta$ it possible to represent in the form of an infinite series
\begin{equation*}
g(x,\alpha,\theta)= \frac{1}{\alpha\pi}\sum_{n=0}^{\infty}\frac{x^n}{n!}\Gamma\left(\frac{n+1}{\alpha}\right) \sin\left(\tfrac{\pi}{2}(n+1)(1-\theta)\right).
\end{equation*}
\end{corollary}

\begin{proof}
We examine the convergence of the series (\ref{eq:gN0}). As we can see, this series is sign-alternating. Consequently
\begin{equation*}
  g_N^0(x,\alpha,\theta)\leqslant|g_N^0(x,\alpha,\theta)|\leqslant \frac{1}{\alpha\pi}\sum_{n=0}^{N-1}\frac{|x|^n}{n!}\Gamma\left(\frac{n+1}{\alpha}\right) \left|\sin\left(\tfrac{\pi}{2}(n+1)(1-\theta)\right)\right|\leqslant
  \frac{1}{\alpha\pi}\sum_{n=0}^{N-1}\frac{|x|^n}{n!}\Gamma\left(\frac{n+1}{\alpha}\right)
\end{equation*}
We apply the Cauchy criterion in the limiting form to the resulting series.
\begin{multline*}
  \lim_{n\to\infty}\left(\frac{|x|^n}{\alpha\pi}\frac{\Gamma\left((n+1)/\alpha\right)}{n!}\right)^{1/n}=
\lim_{n\to\infty}\frac{|x|}{(\alpha\pi)^{1/n}}\left(\frac{\Gamma\left((n+1)/\alpha\right)}{\Gamma(n+1)}\right)^{1/n}= \\
  =|x|\lim_{n\to\infty}\left(\frac{\exp\left\{-\frac{n+1}{\alpha}\right\} \left(\frac{n+1}{\alpha}\right)^{\frac{n+1}{\alpha}-\frac{1}{2}}\sqrt{2\pi}} {\exp\left\{-n-1\right\}(n+1)^{n+1-\frac{1}{2}}\sqrt{2\pi}}\right)^{\frac{1}{n}} =\\
  =|x|\lim_{n\to\infty}e^{\left(1+\frac{1}{n}\right)\left(1-\frac{1}{\alpha}\right)} \alpha^{-\frac{1}{\alpha}-\frac{1}{n}\left(\frac{1}{\alpha}-\frac{1}{2}\right)} (n+1)^{\left(1+\frac{1}{n}\right)\left(\frac{1}{\alpha}-1\right)}
   = |x|e^{1-\frac{1}{\alpha}}\alpha^{-\frac{1}{\alpha}}\lim_{n\to\infty}(n+1)^{\left(\frac{1}{\alpha}-1\right)}=
   \left\{\begin{array}{cc}
            \infty, & \alpha<1, \\
            |x|,&\alpha=1,\\
            0, & \alpha>1.
          \end{array}\right.
\end{multline*}
Here the Stirling’s formula was used
\begin{equation}\label{eq:Stirling}
\Gamma(z)\sim e^{-z}z^{z-\frac{1}{2}}\sqrt{2\pi},\quad z\to\infty,\quad |\arg z|<\pi.
\end{equation}
From the result obtained we can see that for the values $\alpha<1$  the series (\ref{eq:gN0}) diverges for any $x$,  with the value $\alpha=1$  the series (\ref{eq:gN0}) converges for any values of $x$, satisfying the condition $|x|<1$, and in the case $\alpha>1$ the series (\ref{eq:gN0}) converges for any $x$.

We consider the case $\alpha<1$. In this case the series (\ref{eq:gN0}) diverges at $N\to\infty$. However, from the expression (\ref{eq:absRN}) it follows that  for some fixed $N$
\begin{equation*}
  R_N^0=O\left(x^N\right), \quad x\to0.
\end{equation*}
Thus, with every $N$ we have
\begin{equation*}
  g(x,\alpha,\theta)=\frac{1}{\alpha\pi}\sum_{n=0}^{N-1}\frac{x^n}{n!}\Gamma\left(\frac{n+1}{\alpha}\right) \sin\left(\tfrac{\pi}{2}(n+1)(1-\theta)\right)+O\left(x^N\right),\quad x\to0.
\end{equation*}
As a result, we have obtained the definition of an asymptotic series. Consequently,
\begin{equation*}
  g(x,\alpha,\theta)\sim\frac{1}{\alpha\pi}\sum_{n=0}^{N-1}\frac{x^n}{n!}\Gamma\left(\frac{n+1}{\alpha}\right) \sin\left(\tfrac{\pi}{2}(n+1)(1-\theta)\right),\quad x\to0,\quad \alpha<1.
\end{equation*}

We consider the case $\alpha>1$. From the expression (\ref{eq:gx0}) it follows that
\begin{equation}\label{eq:gx0-gN0}
  |g(x,\alpha,\theta)-g_N^0(x,\alpha,\theta)|\leqslant\frac{1}{\alpha\pi}\frac{\Gamma\left(\frac{N+1}{\alpha}\right)}{N!}|x|^N.
\end{equation}
We will set some arbitrary $x$ and consider the limit of the right-hand side of this inequality under the condition $N\to\infty$ . We have
\begin{multline*}
  \frac{1}{\alpha\pi}\lim_{N\to\infty}\frac{\Gamma\left(\frac{N+1}{\alpha}\right)}{N!}|x|^N= \frac{1}{\alpha\pi}\lim_{N\to\infty}\frac{\Gamma\left(\frac{N+1}{\alpha}\right)}{\Gamma(N+1)}|x|^N=
  \frac{1}{\alpha\pi} \lim_{N\to\infty}|x|^N\frac{\exp\left\{-\frac{N+1}{\alpha}\right\} \left(\frac{N+1}{\alpha}\right)^{\frac{N+1}{\alpha}-\frac{1}{2}}\sqrt{2\pi}}{\exp\left\{-N-1\right\} \left(N+1\right)^{N+1-\frac{1}{2}}\sqrt{2\pi}}= \\
  =\frac{\alpha^{-\left(\frac{1}{2}+\frac{1}{\alpha}\right)}}{\pi}\lim_{N\to\infty}|x|^N\alpha^{-N/\alpha} e^{-(N+1)\left(\frac{1}{\alpha}-1\right)} (N+1)^{(N+1)\left(\frac{1}{\alpha}-1\right)}=\\
  =\frac{\alpha^{-\left(\frac{1}{2}+\frac{1}{\alpha}\right)}}{\pi}\lim_{N\to\infty}|x|^N\alpha^{-N/\alpha} e^{(N+1)\left(\frac{1}{\alpha}-1\right)(\ln(N+1)-1)}=0.
\end{multline*}
Thus, the right-hand side (\ref{eq:gx0-gN0}) represents an element of an infinitesimal sequence. This means that the sequence $g_N^0(x,\alpha,\theta)$ at $N\to\infty$ converges to the density $g(x,\alpha,\theta)$.  Therefore, in the case $\alpha>1$ for any fixed $x$ for the density $g(x,\alpha,\theta)$ the representation in the form of an infinite series is valid
\begin{equation*}
  g(x,\alpha,\theta)=\frac{1}{\alpha\pi}\sum_{n=0}^{\infty}\frac{x^n}{n!}\Gamma\left(\frac{n+1}{\alpha}\right) \sin\left(\tfrac{\pi}{2}(n+1)(1-\theta)\right).
\end{equation*}

Now we consider the case $\alpha=1$. From the expression (\ref{eq:gx0}) it directly follows
\begin{equation*}
  |g(x,1,\theta)-g_N^0(x,1,\theta)|\leqslant\frac{|x|^N}{\pi}.
\end{equation*}
We fix some arbitrary $x$ and find the redistribution under the condition $N\to\infty$. As a result, we obtain
\begin{equation*}
  \lim_{N\to\infty}\frac{|x|^N}{\pi}=\left\{\begin{array}{cc}
                                       0, & |x|<1, \\
                                       \infty, & |x|\geqslant1.
                                     \end{array}\right.
\end{equation*}
Thus, the right side of the previous expression at $|x|<1$ is an element of an infinitesimal sequence. Therefore, the sequence $g_N^0(x,1,\theta)$ converges to the density $g(x,1,\theta)$ at $N\to\infty$ and $|x|<1$. Now substituting the value $\alpha=1$, in the series  (\ref{eq:gN0}) we obtain (\ref{eq:gx0_a1}).
\
\begin{flushright}
  $\Box$
\end{flushright}
\end{proof}

The proved corollary shows that in the case $\alpha=1$ in the interval $-1<x<1$ the series (\ref{eq:gx0_a1}) converges to the density $g(x,1,\theta)$. It is important to show that in this case the series (\ref{eq:gx0_a1}) converges to the density (\ref{eq:gx_a1}). We formulate this result in the form

\begin{remark}\label{rm:pdf_a1}
In the case $\alpha=1$ for any $-1<\theta<1$ in the region $-1<x<1$  the series (\ref{eq:gx0_a1}) converges to the density (\ref{eq:gx_a1}).
\end{remark}

\begin{proof}
To prove this remark we will consider the density (\ref{eq:gx_a1}) and show that the expansion of this density in a Taylor series in the vicinity of the point $x=0$ has the form (\ref{eq:gx0_a1}). For the convenience of further presentation, we use the reduction formulas  $\cos\left(\tfrac{\pi}{2}\theta\right)=\sin\left(\tfrac{\pi}{2}(1-\theta)\right)$, $\sin\left(\tfrac{\pi}{2}\theta\right)=\cos\left(\tfrac{\pi}{2}(1-\theta)\right)$ and represent the density (\ref{eq:gx_a1}) in the form
\begin{equation*}
  g(x,1,\theta)=\frac{\sin\left(\tfrac{\pi}{2}(1-\theta)\right)}{\pi\left(x^2-2x\cos\left(\tfrac{\pi}{2}(1-\theta)\right)+1\right)}.
\end{equation*}

Now we expand the density $g(x,1,\theta)$ in a Taylor series in the vicinity of the point $x=0$. Since this density is an infinitely differentiable function, we have
\begin{equation}\label{eq:g(x)_seres}
  g(x,1,\theta)=g(0,1,\theta)+\sum_{n=1}^{\infty}\frac{1}{n!}\left.\frac{d^n g(x,1,\theta)}{dx^n}\right|_{x=0}x^n
\end{equation}
We will draw attention that the function $g(x,1,\theta)$ is a complex function. We will introduce the designations
\begin{equation}\label{eq:f_g_def}
f\equiv f(u)=1/u,\quad g\equiv g(x)=x^2-2x \cos(\tfrac{\pi}{2}(1-\theta))+1.
\end{equation}
In view of the introduced designations, the density (\ref{eq:gx_a1}) takes the form
\begin{equation*}
  g(x,1,\theta)=\frac{\sin(\tfrac{\pi}{2}(1-\theta))}{\pi}f(g(x)).
\end{equation*}
To calculate the $n$-th derivative we use the Bruno formula
\begin{equation}\label{eq:dg(x)dx}
  \frac{d^n g(x,1,\theta)}{dx^n}=\frac{\sin(\tfrac{\pi}{2}(1-\theta))}{\pi}\frac{d^n f(g(x))}{dx^n}=\frac{\sin(\tfrac{\pi}{2}(1-\theta))}{\pi}Y_n(fg_1,fg_2,\dots,fg_n),
\end{equation}
where $Y_n(fg_1,fg_2,\dots,fg_n)$ are the Bell polynomials (see~\cite{Riordan1958_en})
\begin{equation}\label{eq:BellPolinom}
  Y_n(fg_1,fg_2,\dots,fg_n)=\sum \frac{n!f_m}{k_1!k_2!\dots k_n!}\left(\frac{g_1}{1!}\right)^{k_1} \left(\frac{g_2}{2!}\right)^{k_2}\dots \left(\frac{g_n}{n!}\right)^{k_n}.
\end{equation}
Here $f^m\equiv f_m$, $m=k_1+k_2+\dots+k_n$, the sum is taken over all solutions to the equation
\begin{equation}\label{eq:BelPolSumEq}
k_1+2k_2+\dots+n k_n=n,\quad k_j\geqslant0,\quad j=1,2,\dots,n.
\end{equation}
and
\begin{equation*}
  f_m=\left.\frac{d^mf(u)}{du^m}\right|_{u=g(x)},\quad g_j=\frac{d^j g(x)}{dx^j }.
\end{equation*}
Taking into account (\ref{eq:f_g_def}), we get
\begin{equation}\label{eq:fk_g}
  f_m=(-1)^m\left.\frac{m!}{u^{m+1}}\right|_{u=g(x)}=\frac{(-1)^m m!}{\left(x^2-2x \cos\left(\tfrac{\pi}{2}(1-\theta)\right)+1\right)^{m+1}}.
\end{equation}
For coefficients $g_j$ we have
\begin{equation}\label{eq:gk_g}
  g_1=2x-2\cos\left(\tfrac{\pi}{2}(1-\theta)\right),\quad g_2=2,\quad g_3
  =g_4=\dots=g_n=0.
\end{equation}
This shows that in the expression (\ref{eq:BellPolinom}) the sum contains the summands that satisfy the equation
\begin{equation}\label{eq:BelPolSumEq_g}
k_1+2k_2=n.
\end{equation}

Indeed, in the expression (\ref{eq:BellPolinom}) the summation is done over all solutions to the equation (\ref{eq:BelPolSumEq}). In case, if the solution $k_j\neq0,\ j=3,4,\dots,n$, then the corresponding term in the sum will be equal to zero, since $g_j=0$, $j=3,4,\dots,n$. If $k_j=0$, $j=3,4,\dots,n$, then the multiplier $(g_j/j!)^{k_j}=1$, since  $0^0=1$.  Consequently, in the expression (\ref{eq:BellPolinom}) there are the summands that satisfy the solution to the equation (\ref{eq:BelPolSumEq_g}). This significantly simplifies the summation. It follows from the equation (\ref{eq:BelPolSumEq_g}) that $k_1=n-2k_2$. Taking into consideration that $k_1\geqslant0$ and $k_2\geqslant0$, we obtain $k_2=0,1,2,\dots,\left[\tfrac{n}{2}\right]$, where $[A]$ means the integer part of the number $A$. It gives an opportunity to introduce directly the summation index in the sum (\ref{eq:BellPolinom}). In view of the foregoing, the formula (\ref{eq:BellPolinom}) takes the form
\begin{equation*}
  Y_n(fg_1,fg_2)=\sum_{k=0}^{\left[\frac{n}{2}\right]}\frac{n!f_{n-k}}{(n-2k)!k!}\left(\frac{g_1}{1!}\right)^{n-2k} \left(\frac{g_2}{2!}\right)^k,
\end{equation*}
where the relation $k_1=n-2k_2$ is used and  the summation index $k\equiv k_2$ is introduced. Now substituting this relation in (\ref{eq:dg(x)dx}) and using (\ref{eq:fk_g}) and (\ref{eq:gk_g}), we obtain
\begin{equation}\label{eq:g_n_der}
  \frac{d^n g(x,1,\theta)}{dx^n}=\frac{\sin\left(\frac{\pi}{2}(1-\theta)\right)}{\pi} \sum_{k=0}^{\left[\tfrac{n}{2}\right]} \frac{(-1)^{n-k} n! (n-k)!}{k!(n-2k)!}\frac{\left(2x-2\cos\left(\tfrac{\pi}{2}(1-\theta)\right)\right)^{n-2k}} {\left(x^2-2x\cos\left(\tfrac{\pi}{2}(1-\theta)\right)+1\right)^{n-k+1}}.
\end{equation}

Now we calculate the value of this derivative in the point $x=0$. It is easy to see that
\begin{equation}\label{eq:dg(0)dx}
  \left.\frac{d^n g(x,1,\theta)}{dx^n}\right|_{x=0}=\frac{\sin\left(\tfrac{\pi}{2}(1-\theta)\right)}{\pi} \sum_{k=0}^{\left[\frac{n}{2}\right]} \frac{(-1)^k n!(n-k!)}{k!(n-2k)!}\left(2\cos\left(\tfrac{\pi}{2}(1-\theta)\right)\right)^{n-2k},
\end{equation}
where it is taken into account that $(-1)^{2n-3k}=(-1)^k$.

Next, we use the general formula for $\sin(n\varphi)$ (see, for example, \cite{Schaum2018})
\begin{equation}\label{eq:sin_na}
  \sin(n\varphi)=\sin\varphi\sum_{k=1}^{\left[\frac{n-1}{2}\right]}(-1)^k\frac{(n-k-1)!}{k!(n-2k-1)!}(2\cos\varphi)^{n-2k-1}.
\end{equation}
Using this formula in (\ref{eq:dg(0)dx}), we obtain
\begin{equation*}
  \left.\frac{d^n g(x,1,\theta)}{dx^n}\right|_{x=0}=\frac{n!}{\pi}\sin\left(\tfrac{\pi}{2}(n+1)(1-\theta)\right).
\end{equation*}
Using this expression now in (\ref{eq:g(x)_seres}), we get
\begin{equation*}
  g(x,1,\theta)=g(0,1,\theta)+\frac{1}{\pi}\sum_{n=1}^{\infty}x^n\sin(\tfrac{\pi}{2}(n+1)(1-\theta))=
  \frac{1}{\pi}\sum_{n=0}^{\infty}\sin(\tfrac{\pi}{2}(n+1)(1-\theta)) x^n.
\end{equation*}
Thus, the expansion of the density (\ref{eq:gx_a1}) into an infinite Taylor series in the vicinity of the point $x=0$ agrees exactly with the series (\ref{eq:gx0_a1}). It completely proves the corollary.
\begin{flushright}
  $\Box$
\end{flushright}
\end{proof}

Theorem~\ref{th:gx0_expan} gives an opportunity to determine the range of values $x$ within which the absolute error of the density calculation using the series (\ref{eq:gN0}) and for some fixed $N$ will not exceed the predetermined value $\varepsilon$. This turns out to be very convenient when calculating the probability density. Indeed, from the relations (\ref{eq:gx0}) and (\ref{eq:absRN}) we obtain
\begin{equation*}
  |g(x,\alpha,\theta)-g_N^0(x,\alpha,\theta)|\leqslant\frac{1}{\alpha\pi}\frac{|x|^N}{N!}\Gamma\left(\frac{N+1}{\alpha}\right)
\end{equation*}
If, for a given value of $N$ we set the absolute value of the error $|g(x,\alpha,\theta)-g_N^0(x,\alpha,\theta)|=\varepsilon$,  then it becomes possible to introduce the threshold coordinate
\begin{equation}\label{eq:x_eps}
  x_\varepsilon^N=\left(\frac{\alpha\pi\varepsilon N!}{\Gamma\left(\frac{N+1}{\alpha}\right)}\right)^{\frac{1}{N}}
\end{equation}
This value shows that for coordinates $|x|\leqslant|x_\varepsilon^N|$ the absolute value of the density calculation error using the series (\ref{eq:gN0}) will not exceed $\varepsilon$, i.e.
\begin{equation*}
  |g(x,\alpha,\theta)-g_N^0(x,\alpha,\theta)|\leqslant\varepsilon,\quad -x_\varepsilon^N\leqslant x\leqslant x_\varepsilon^N.
\end{equation*}
Thus, to calculate the probability density, we can use  the $N$-th partial sum (\ref{eq:gN0}) in the range of coordinates $x\in[-x_\varepsilon^N, x_\varepsilon^N]$. In this case, the magnitude of the absolute error at  fixed $N$ will not exceed the chosen value $\varepsilon$.

\begin{figure}
  \centering
  \includegraphics[width=0.45\textwidth]{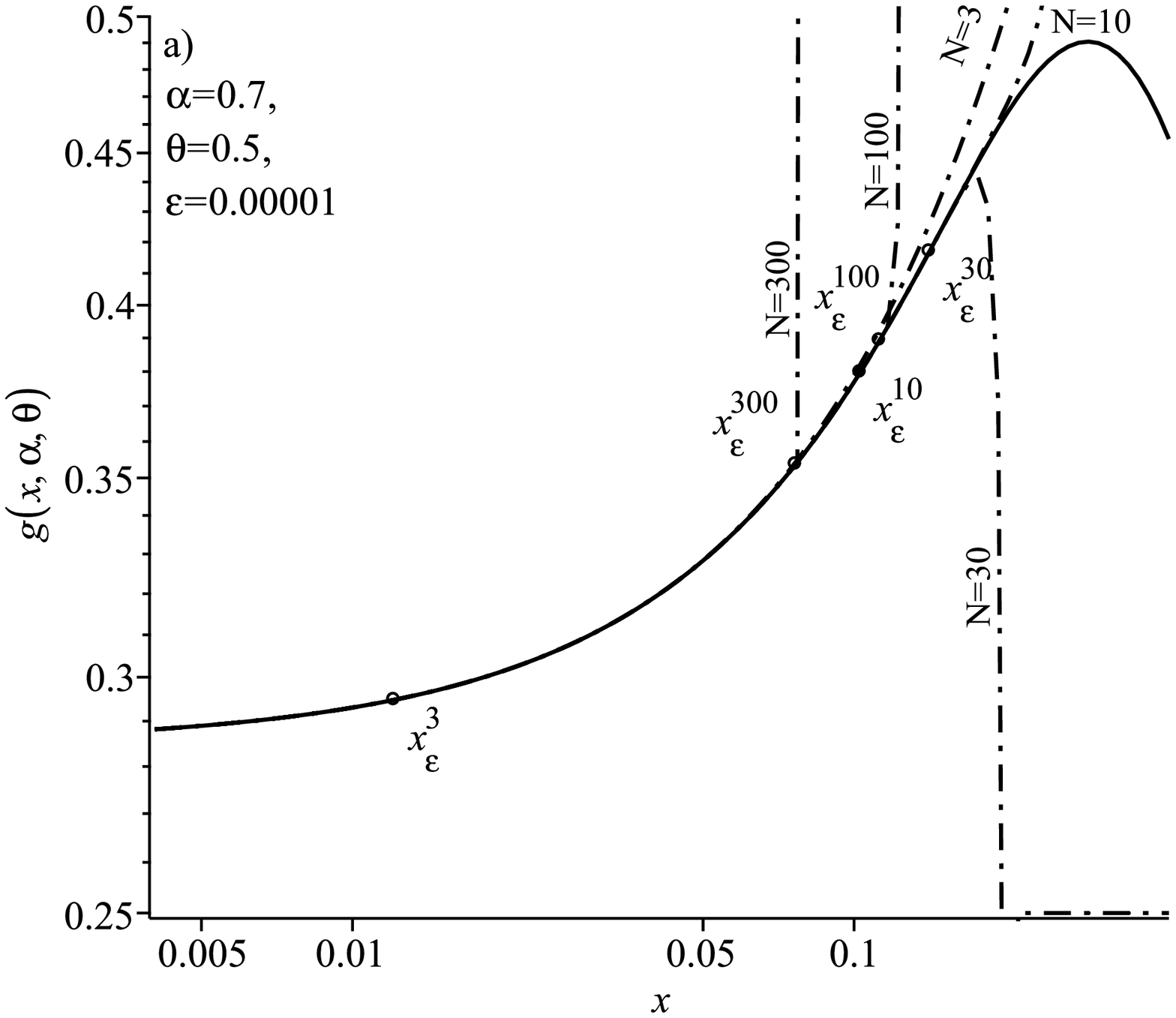}\hfill
  \includegraphics[width=0.45\textwidth]{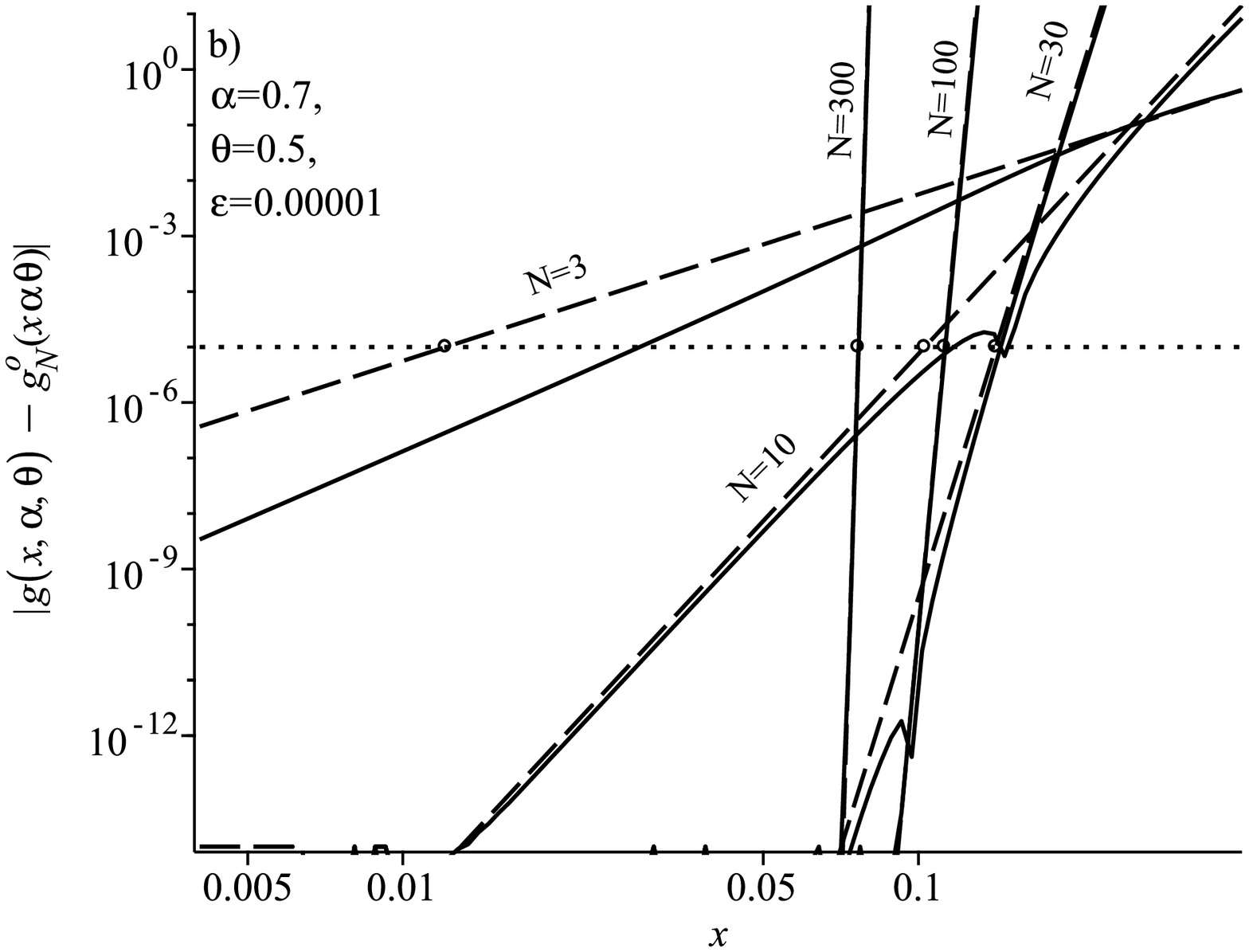}\hfill
  \caption{(a) Probability density $g(x,\alpha,\theta)$ for the parameter values shown in the figure.  The solid curve is an integral representation (\ref{eq:pdfI}),  dashed curves are a series (\ref{eq:gN0}) for different values of the number of  $N$ summands in total. Circles are the position of the threshold coordinate (\ref{eq:x_eps}) for each value $N$. (b) The graph of the absolute error of calculating the probability density using a series (\ref{eq:gN0}) for the case given in Figure (а).  Solid curves are the exact value of the absolute error $|g(x,\alpha,\theta)-g_N^0(x,\alpha,\theta)|$, dashed curves are the estimate (\ref{eq:absRN}), dotted line  shows the position of the specified accuracy level  $\varepsilon$. Circles - demonstrate the location of the threshold coordinate (\ref{eq:x_eps}) for the specified values $N$}\label{fig:a07x0}
\end{figure}
\begin{figure}
  \centering
  \includegraphics[width=0.45\textwidth]{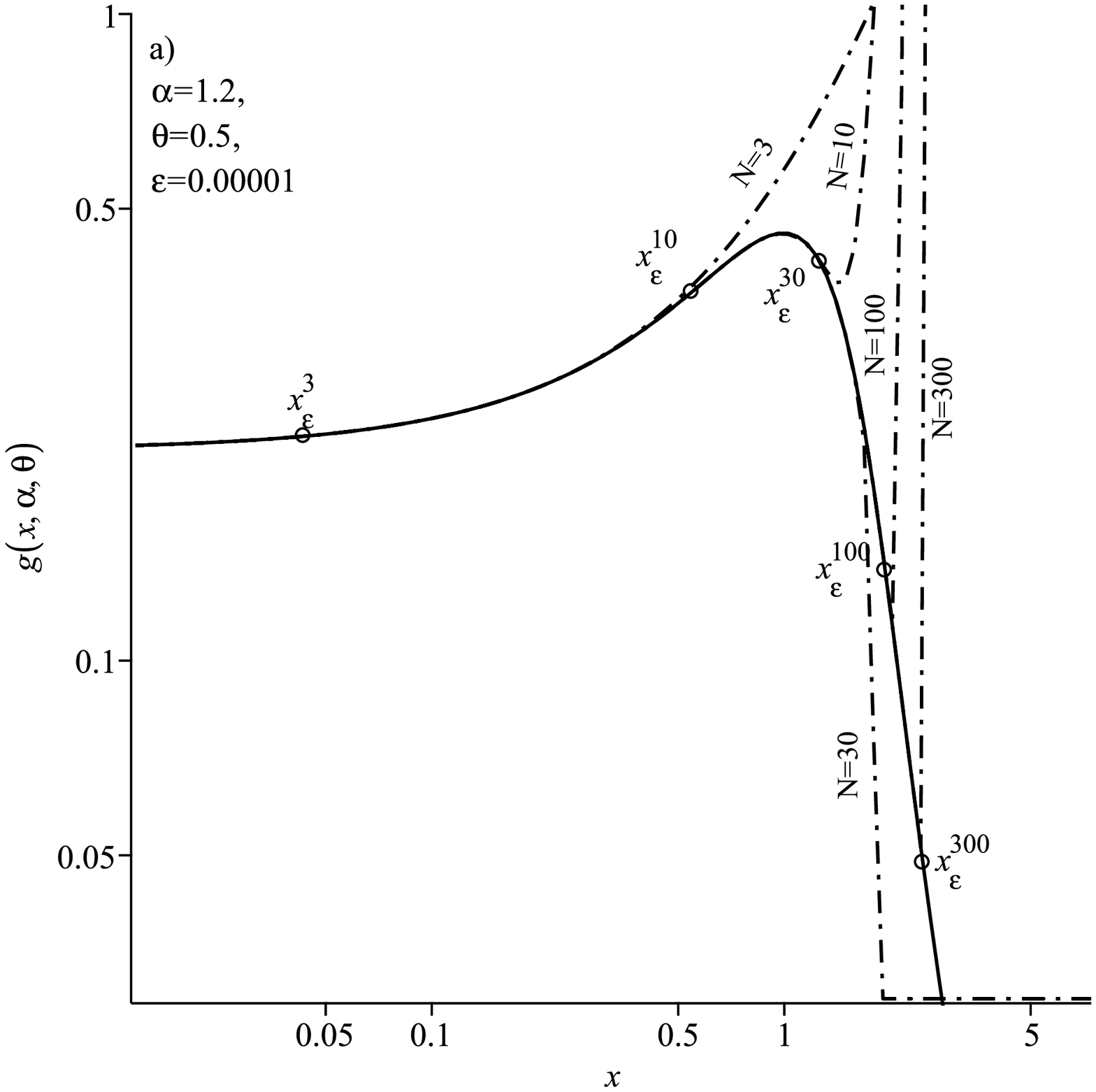}\hfill
  \includegraphics[width=0.45\textwidth]{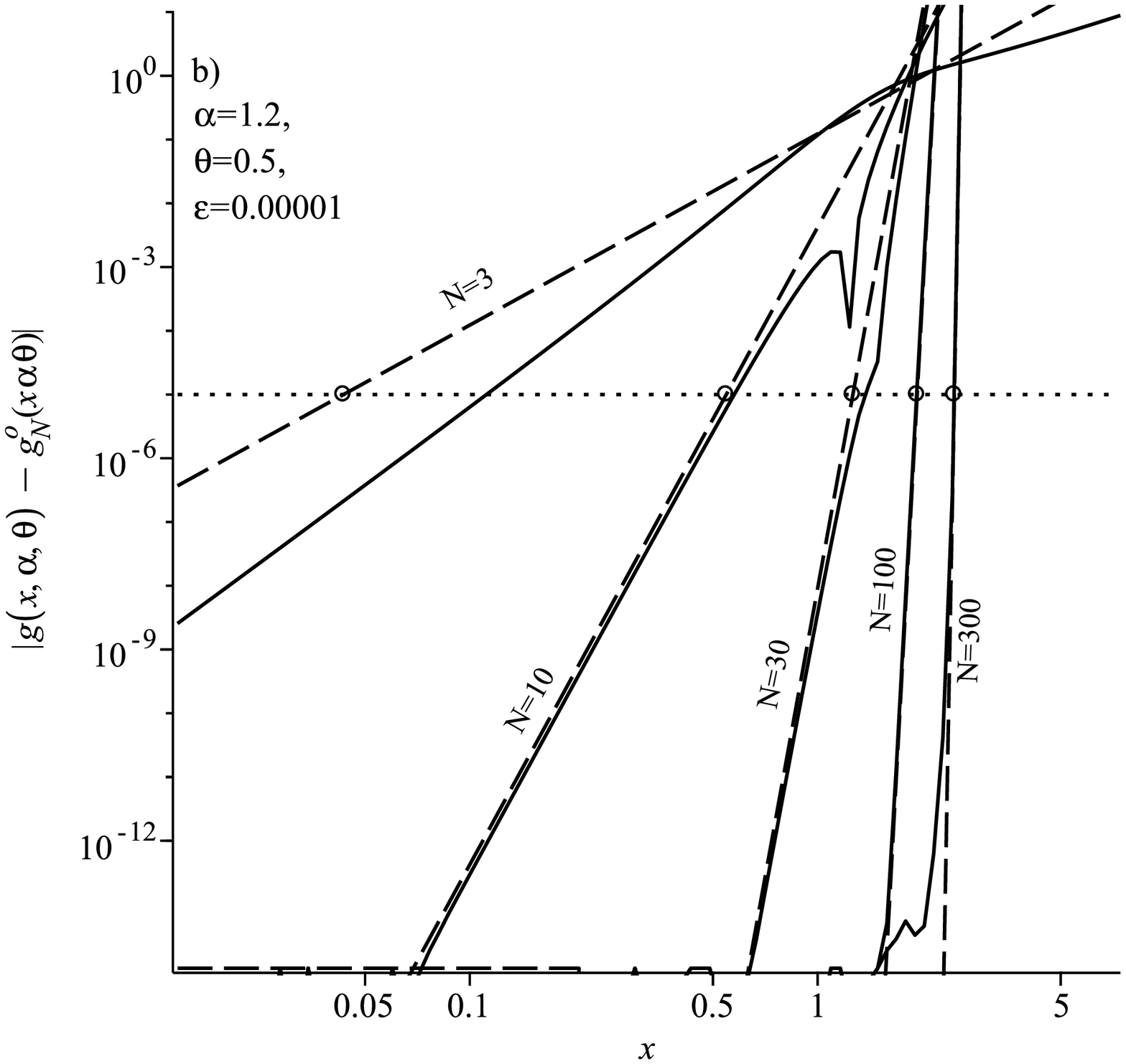}\hfill
  \caption{(a) Probability density $g(x,\alpha,\theta)$ for the parameter values shown in the figure.  A solid curve is the integral representation (\ref{eq:pdfI}),  dashed curves are a series (\ref{eq:gN0}) for different values of the number of $N$ summands in total. Circles are the position of the threshold coordinate (\ref{eq:x_eps}) for each value of $N$. (b) The graph of the absolute error of calculating the probability density using a series (\ref{eq:gN0}) for the case given in Figure (а).  Solid curves are the exact value of the absolute error $|g(x,\alpha,\theta)-g_N^0(x,\alpha,\theta)|$, dashed curves are the estimate (\ref{eq:absRN}), dotted line  shows the position of the specified accuracy level  $\varepsilon$. Circles  - demonstrate the location of the threshold coordinate (\ref{eq:x_eps}) for the specified values of $N$}\label{fig:a12x0}
\end{figure}

Figures~\ref{fig:a07x0}a~and~\ref{fig:a12x0}a show the results of calculating the probability density $g(x,\alpha,\theta)$ using the series (\ref{eq:gN0}).  In these figures, the solid curve corresponds to the exact density values $g(x,\alpha,\theta)$, calculated with the help of (\ref{eq:pdfI}), the dashed-dotted curve corresponds to the density calculation results using the series (\ref{eq:gN0}) for the specified values of $N$.  Figures~\ref{fig:a07x0}b~and~\ref{fig:a12x0}b show the results of the calculation of the absolute error  $|g(x,\alpha,\theta)-g_N^0(x,\alpha,\theta)|$. In these figures, the solid curves correspond to the  exact value of the absolute error $|g(x,\alpha,\theta)-g_N^0(x,\alpha,\theta)|$, where $g(x,\alpha,\theta)$ -- the exact density value calculated when using (\ref{eq:pdfI}), $g_N^0(x,\alpha,\theta)$ - the series (\ref{eq:gN0}), the dashed curve corresponds to the estimate of the remainder term (\ref{eq:absRN}). The calculation results are given for the specified values of $N$ in the figures. In all these figures, the circles show the position of the threshold coordinate $x_\varepsilon^N$ for the selected level of accuracy $\varepsilon$  and each number of summands $N$. It is clear from Figure~\ref{fig:a07x0}b~and~\ref{fig:a12x0}b that in the region $|x|\leqslant x_\varepsilon^N$ the absolute magnitude of the error does not exceed the specified level of accuracy $\varepsilon$ for all  $N$. This means that at $|x|\leqslant x_\varepsilon^N$ the expansion (\ref{eq:gN0}) can be used to calculate the density.

Corollary~\ref{corol:gN0_a><1}  shows that in the case $\alpha<1$ the series (\ref{eq:gN0}) is divergent at $N\to\infty$, and in the case $\alpha>1$ this series converges. The cause of this behavior lies in the ratio $\Gamma((n+1)/\alpha)/\Gamma(n+1)$, which is present in this series. At $\alpha<1$ this ratio turns out to be more than unity and, as $n$ increases, this ratio only rises. Therefore, to achieve the  specified calculation accuracy $\varepsilon$ one has to decrease the value of $x$. It is clearly seen from the behavior of the threshold coordinates $x_\varepsilon^N$.  Figures~\ref{fig:a07x0}a~and~\ref{fig:a07x0}b show that  at first the addition of summands in the expansion (\ref{eq:gN0}) leads to an increase in the range of $x$ for which the inequality $|g(x,\alpha,\theta-g_N^0(x,\alpha,\theta))|\leqslant\varepsilon$  is satisfied.  The fact that $x_\varepsilon^3<x_\varepsilon^{10}<x_\varepsilon^{30}$ testifies to it. However, further addition of summands leads to an increase in the ratio $\Gamma((n+1)/\alpha)/\Gamma(n+1)$ and, thus, to an increase in the absolute calculation error. Therefore, to achieve the specified level of accuracy, it is necessary to decrease the value of the coordinate $x$. This causes the threshold coordinate $x_\varepsilon^N$ to start decreasing and we see that $x_\varepsilon^{100}>x_\varepsilon^{300}$.

In the case $\alpha>1$ the situation changes. In this case the ratio $\Gamma((n+1)/\alpha)/\Gamma(n+1)<1$ and as $n$ increases this ratio only decreases. Consequently, the increase in the number of summands in (\ref{eq:gN0}) increases the accuracy of the density calculation. This leads to the fact that the range of values of $x$, for which the inequality $|g(x,\alpha,\theta)-g_N^0(x,\alpha,\theta)|\leqslant\varepsilon$ is met increases with the addition of the number of summands $N$ in the sum (\ref{eq:gN0}). This is clearly seen from the location of  the threshold coordinates $x_\varepsilon^N$, shown in Figures~\ref{fig:a12x0}a~and~\ref{fig:a12x0}b. We can see from the figures that $x_\varepsilon^3<x_\varepsilon^{10}<x_\varepsilon^{30}<x_\varepsilon^{100}<x_\varepsilon^{300}$. Thus, in the case $\alpha>1$ the series (\ref{eq:gN0}) is convergent for all $x$ at $N\to\infty$.

\begin{figure}
  \centering
  \includegraphics[width=0.45\textwidth]{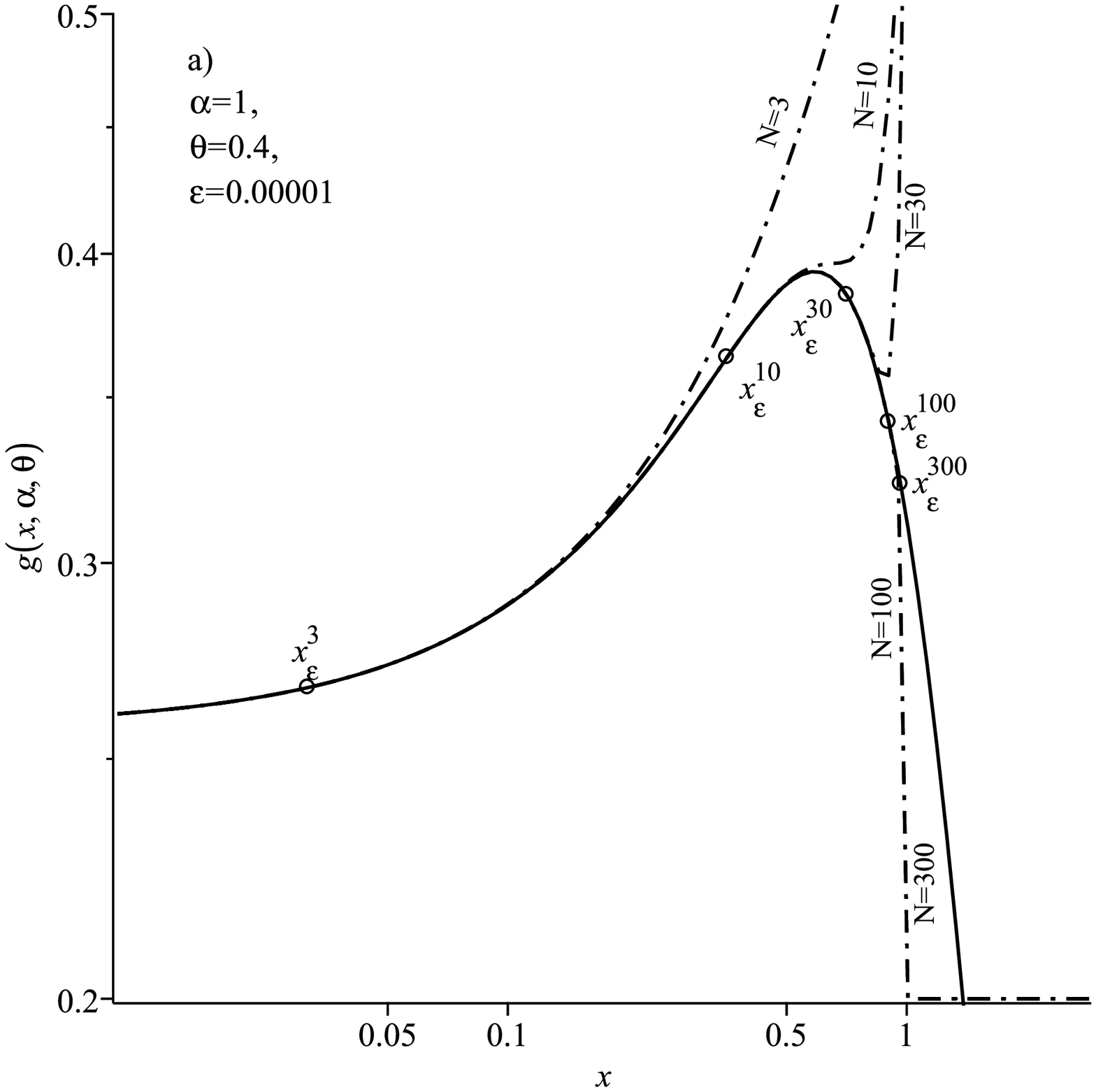}\hfill
  \includegraphics[width=0.45\textwidth]{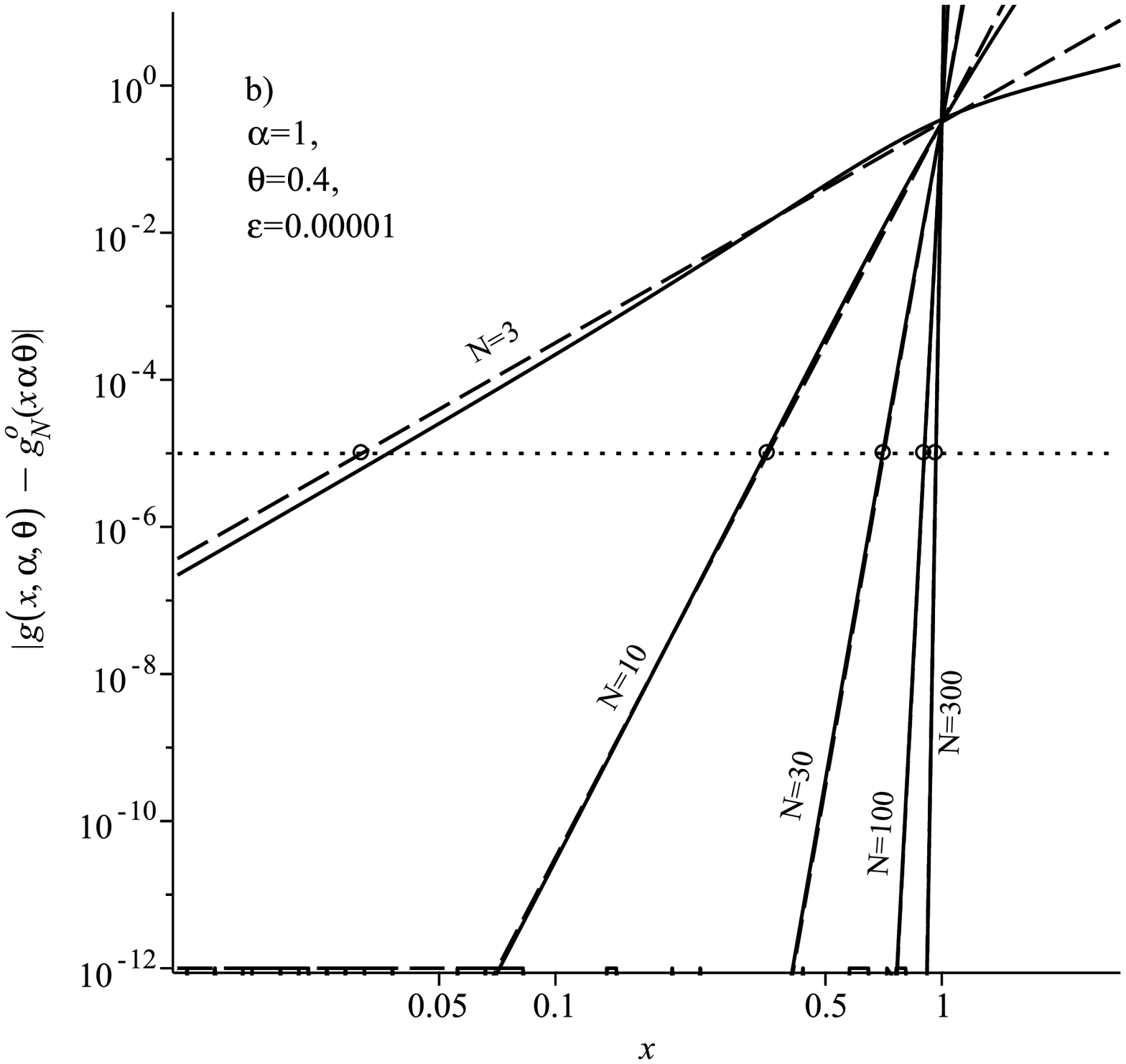}\hfill
  \caption{(a) Probability density $g(x,\alpha,\theta)$ for the parameter values shown in the figure.  A solid curve is the formula (\ref{eq:gx_a1}),  dashed curves are the series (\ref{eq:gN0}) for different values of the number of $N$ summands in total. Circles are the position of the threshold coordinate (\ref{eq:x_eps}) for each value of $N$. (b) The graph of the absolute error of calculating the probability density using a series (\ref{eq:gN0}) for the case given in Figure (а).  Solid curves are the exact value of the absolute error $|g(x,\alpha,\theta)-g_N^0(x,\alpha,\theta)|$, dashed curves are the estimate (\ref{eq:absRN}), dotted line  - shows the position of the specified accuracy level $\varepsilon$. Circles – demonstrate the location of the threshold coordinate (\ref{eq:x_eps}) for the specified values of $N$}\label{fig:a1x0}
\end{figure}

The results of calculations for the case $\alpha=1$ are given in Fig.~\ref{fig:a1x0}.  This figure shows, $x_\varepsilon^{3}<x_\varepsilon^{10}<x_\varepsilon^{30}<x_\varepsilon^{100}<x_\varepsilon^{300}$.  Thus, an increase in the number of summands $N$ in the sum (\ref{eq:gN0}) leads to an increase in the interval $x$, within which the inequality $|g(x,\alpha,\theta)-g_N^0(x,\alpha,\theta)|\leqslant\varepsilon$ is satisfied. Here, the formula (\ref{eq:gx_a1}) was used to calculate the density. It is also seen from the figure that for all presented $N$ at $x\to1$ the partial sums of the series diverge. This is in full agreement with the statement of corollary~\ref{corol:gN0_a><1} which states that in the case $\alpha=1$ the series (\ref{eq:gN0}) converges at $|x|<1$.

\section{Representation of the distribution function in the form of a power series}

Now we will try to obtain the representation of the distribution function in the case $x\to0$ in the form of a power series. We will formulate the result obtained as a theorem

\begin{theorem}\label{th:Gx0_expan}
In the case $x\to0$ for any admissible set of parameters $(\alpha,\theta)$, except for the values $\alpha=1,\theta=\pm1$ for the distribution function $G(x,\alpha,\theta)$ a representation in the form of a power series is valid.
\begin{equation}\label{eq:Gx0}
  G(x,\alpha,\theta)=\frac{1-\theta}{2}+G_N^0(x,\alpha,\theta)+\mathcal{R}_N^0(x,\alpha,\theta),\quad x\to0
\end{equation}
where
\begin{align}
  G_N^0(x,\alpha,\theta) & =\frac{1}{\pi\alpha}\sum_{n=0}^{N-1}\frac{x^{n+1}}{(n+1)!}\Gamma\left(\frac{n+1}{\alpha}\right) \sin\left(\tfrac{\pi}{2}(n+1)(1-\theta)\right), \label{eq:GN0}\\
  |\mathcal{R}_n^0(x,\alpha,\theta)| & \leqslant\frac{|x|^{N+1}}{\alpha\pi (N+1)!}\Gamma\left(\frac{N+1}{\alpha}\right).\label{eq:RGN0}
\end{align}
\end{theorem}
\begin{proof}

From the definition of the distribution function, it follows
  \begin{equation*}
    G(x,\alpha,\theta)-G(0,\alpha,\theta)=\int_{0}^{x}g(\xi,\alpha,\theta)d\xi.
  \end{equation*}
Here $G(0,\alpha,\theta)$ is the value of the distribution function in the point $x=0$ and is determined by the formula (\ref{eq:G_x=0}). Using the expansion (\ref{eq:g_expan}) for the density $g(x,\alpha,\theta)$ we obtain
\begin{equation}\label{eq:G_expan_tmp0}
  G(x,\alpha,\theta)=\frac{1-\theta}{2} +G_N^0(x,\alpha,\theta)+\mathcal{R}_N^0(x,\alpha,\theta),\quad x\to0
\end{equation}
where
\begin{align}
  G_N^0(x,\alpha,\theta)&=\int_{0}^{x}g_N^0(\xi,\alpha,\theta)d\xi,\label{eq:GN0_tmp0} \\
  \mathcal{R}_N^0(x,\alpha,\theta)&=\int_{0}^{x}R_n^0(\xi,\alpha,\theta)d\xi.\label{eq:RGN0_tmp0}
\end{align}
Here $g_N^0(x,\alpha,\theta)$ and $R_N^0(x,\alpha,\theta)$ are determined by the expressions  (\ref{eq:gN0}) and (\ref{eq:RNx0_tmp1}) respectively.

To calculate the partial sum $G_N^0(x,\alpha,\theta)$ we will make use of the results of  theorem~\ref{th:gx0_expan}. It was obtained in this theorem that the partial sum $g_N^0(x,\alpha,\theta)$ has the form (\ref{eq:gN0}).  Substituting the expression (\ref{eq:gN0})  in (\ref{eq:GN0_tmp0}) and changing the order of integration and summation, we get
\begin{multline}\label{eq:GN0_tmp1}
  G_N^0(x,\alpha,\theta)=\frac{1}{\pi\alpha}\sum_{n=0}^{N-1}\frac{1}{n!}\Gamma\left(\frac{n+1}{\alpha}\right) \sin\left(\tfrac{\pi}{2}(n+1)(1-\theta)\right)\int_{0}^{x}\xi^kd\xi=\\
  =\frac{1}{\pi\alpha}\sum_{n=0}^{N-1}\frac{x^{n+1}}{(n+1)!}\Gamma\left(\frac{n+1}{\alpha}\right) \sin\left(\tfrac{\pi}{2}(n+1)(1-\theta)\right).
\end{multline}
Now we obtain the expression for the remainder $\mathcal{R}_N^0(x,\alpha,\theta)$. Substituting the expression (\ref{eq:RNx0_tmp1}) in (\ref{eq:RGN0_tmp0}) and changing the order of integration, we obtain
\begin{multline*}
  \mathcal{R}_N^0(x,\alpha,\theta)=\int_{0}^{x}R_N^0(\xi,\alpha,\theta)d\xi =
  \frac{1}{\pi}\Re\int_{0}^{\infty}\exp\left\{-t^\alpha \exp\left\{i\tfrac{\pi}{2}\alpha\theta\right\}\right\}dt\int_{0}^{x}R_N(it\xi)d\xi=\\
  \frac{1}{\pi}\Re\int_{0}^{\infty}\exp\left\{-t^\alpha \exp\left\{i\tfrac{\pi}{2}\alpha\theta\right\}\right\}\frac{(it)^N}{N!}\int_{0}^{x}\xi^N e^{it\xi\zeta}d\xi,
\end{multline*}
where $0<\zeta<1$. This integral cannot be calculated directly, since the exact value of $\zeta$ is not known. It is only known that $\zeta\in(0,1)$. However, one can obtain an estimate of this integral.

To obtain an estimate for the integral, we use the inequality $\left|\int_{0}^{x}\xi^N e^{it\xi\zeta}d\xi\right|\leqslant \left|\int_{0}^{x}\xi^N d\xi\right|$. As a result, we get
\begin{multline}\label{eq:RGN0_tmp1}
  |\mathcal{R}_N^0(x,\alpha,\theta)|=\frac{1}{\pi}\left|\Re\int_{0}^{\infty}\exp\left\{-t^\alpha \exp\left\{i\tfrac{\pi}{2}\alpha\theta\right\}\right\}\frac{(it)^N}{N!}\int_{0}^{x}\xi^N e^{it\xi\zeta}d\xi,\right|\leqslant\\
  \leqslant \frac{1}{\pi N!}\left|\Re i^N\int_{0}^{\infty}\exp\left\{-t^\alpha \exp\left\{i\tfrac{\pi}{2}\alpha\theta\right\}\right\} t^N \int_{0}^{x}\xi^Nd\xi\right|=\\
   =\frac{1}{\pi (N+1)!}\left|\Re i^N\int_{0}^{\infty}\exp\left\{-t^\alpha \exp\left\{i\tfrac{\pi}{2}\alpha\theta\right\}\right\} t^N  x^{N+1}\right|\leqslant\\
  \leqslant \frac{|x|^{N+1}}{\alpha \pi(N+1)!}\left|\Re i^N\int_{0}^{\infty}\exp\left\{-\tau\exp\left\{i\tfrac{\pi}{2}\alpha\theta\right\}\right\}\tau^{{\frac{N+1}{\alpha}}-1}d\tau\right|=\\ = \frac{|x|^{N+1}}{\alpha \pi(N+1)!}\Gamma\left(\frac{N+1}{\alpha}\right)\left|\Re i^N \exp\left\{-i\tfrac{\pi}{2}\theta(N+1)\right\} \right|\leqslant\\
  \leqslant\frac{|x|^{N+1}}{\alpha\pi (N+1)!}\Gamma\left(\frac{N+1}{\alpha}\right).
\end{multline}
Here, to calculate the outer integral, the integration variable $t^\alpha=\tau$ was first substituted, and then the formula (\ref{eq:Gamma_intRepr1}) was used. It should be noted that the case $\alpha=1$, $\theta=\pm1$ must be excluded from consideration. Indeed, for such parameter values, the argument $\tfrac{\pi}{2}\alpha\theta=\pm\tfrac{\pi}{2}$ and integral  (\ref{eq:Gamma_intRepr1}) will diverge.  Now substituting the expressions (\ref{eq:GN0_tmp1}) and (\ref{eq:RGN0_tmp1})  in (\ref{eq:GN0_tmp1}) we get the statement of the theorem.

\begin{flushright}
  $\Box$
\end{flushright}
\end{proof}

The proved theorem shows that in the vicinity of the point $x=0$ the expansion (\ref{eq:Gx0}) is valid for the distribution function of a strictly stable law with the characteristic function (\ref{eq:CF_formC}). However, as in the case of the probability density, the obtained power series diverges for all $x$ at $\alpha<1$, and in the case $\alpha>1$ is convergent for all $x$. In the case $\alpha=1$ this series converges at $|x|<1$ and diverges at $|x|\geqslant1$. In this regard, for the values $\alpha<1$ the representation (\ref{eq:Gx0}) is asymptotic, and for the values $\alpha>1$ the expansion $G(x,\alpha,\theta)$ can be represented in the form of an infinite power series. We formulate this result as a corollary.

\begin{corollary}\label{corol:GN0_a><1}
In the case $\alpha<1$ the series (\ref{eq:GN0}) diverges for all $x$ at $N\to\infty$. In this case the asymptotic expansion is valid for the distribution function $G(x,\alpha,\theta)$ for any admissible $\theta$
\begin{equation*}
  G(x,\alpha,\theta)\sim \frac{1-\theta}{2}+\frac{1}{\pi\alpha}\sum_{n=0}^{N-1}\frac{x^{n+1}}{(n+1)!}\Gamma\left(\frac{n+1}{\alpha}\right) \sin\left(\tfrac{\pi}{2}(n+1)(1-\theta)\right), \quad x\to0.
\end{equation*}

In the case $\alpha=1$ the series (\ref{eq:Gx0}) converges at $|x|<1$. In this case the distribution function $G(x,1,\theta)$ for any $\theta\neq\pm1$ can be represented as an infinite series
\begin{equation}\label{eq:Gx0_a1}
  G(x,1,\theta)=\frac{1-\theta}{2}+\frac{1}{\pi}\sum_{n=0}^{\infty}\frac{x^{n+1}}{n+1}\sin\left(\tfrac{\pi}{2}(n+1)(1-\theta)\right).
\end{equation}

In the case $\alpha>1$ the series (\ref{eq:GN0}) at $N\to\infty$ converges for any $x$. In this case the representation in the form of an infinite power series  is true for the distribution function $G(x,\alpha,\theta)$ for any admissible $\theta$
\begin{equation*}
  G(x,\alpha,\theta)=\frac{1-\theta}{2}+\frac{1}{\pi\alpha}\sum_{n=0}^{\infty}\frac{x^{n+1}}{(n+1)!} \Gamma\left(\frac{n+1}{\alpha}\right)\sin\left(\frac{\pi}{2}(n+1)(1-\theta)\right).
\end{equation*}
\end{corollary}

\begin{proof}
We examine the convergence of the series (\ref{eq:GN0}). It is clear that this series is sign-alternating. Consequently
  \begin{multline*}
    G_N^0(x,\alpha,\theta)\leqslant|G_N^0(x,\alpha,\theta)|\leqslant\\ \frac{1}{\pi\alpha}\sum_{n=0}^{N-1}\frac{|x|^{n+1}}{(n+1)!}\Gamma\left(\frac{n+1}{\alpha}\right) \left|\sin\left(\tfrac{\pi}{2}(n+1)(1-\theta)\right)\right|\leqslant
    \frac{1}{\pi\alpha}\sum_{n=0}^{N-1}\frac{|x|^{n+1}}{(n+1)!}\Gamma\left(\frac{n+1}{\alpha}\right).
  \end{multline*}
We apply the Cauchy criterion in the limiting form to the obtained series. Using Stirling’s formula (\ref{eq:Stirling}) and taking into consideration that $n+2\approx n+1$ at $n\to\infty$, we get
  \begin{multline*}
\lim_{n\to\infty}\left(\frac{|x|^{n+1}}{\alpha\pi}\frac{\Gamma\left(\frac{n+1}{\alpha}\right)}{(n+1)!}\right)^{1/n}=
\lim_{n\to\infty}\left(\frac{|x|^{n+1}}{\alpha\pi}\frac{\Gamma\left(\frac{n+1}{\alpha}\right)}{\Gamma(n+2)}\right)^{1/n}=\\
     = \lim_{n\to\infty}\frac{|x|^{1+\frac{1}{n}}}{(\alpha\pi)^{1/n}} \frac{\left(\exp\left\{-\frac{n+1}{\alpha}\right\}\left(\frac{n+1}{\alpha}\right)^{\frac{n+1}{\alpha}-\frac{1}{2}} \sqrt{2\pi}\right)^{\frac{1}{n}}} {\left(\exp\left\{-(n+2)\right\}(n+2)^{n+2-\frac{1}{2}}\sqrt{2\pi}\right)^{\frac{1}{n}}}=\\
     \lim_{n\to\infty}\frac{|x|^{1+\frac{1}{n}}}{(\alpha\pi)^{1/n}}
     \frac{\exp\left\{-\frac{1}{\alpha}\left(1+\frac{1}{n}\right)\right\} \left(\frac{n+1}{\alpha}\right)^{\frac{1}{\alpha}\left(1+\frac{1}{n}\right)-\frac{1}{2n}}} {\exp\left\{-1-\frac{1}{2n}\right\}(n+2)^{1-\frac{3}{2n}}}= \lim_{n\to\infty}|x|e^{1-\frac{1}{\alpha}}\alpha^{-\frac{1}{\alpha}}\frac{(n+1)^\frac{1}{\alpha}}{n+2}=\\
     =\lim_{n\to\infty}|x|e^{1-\frac{1}{\alpha}}\alpha^{-\frac{1}{\alpha}} (n+1)^{\frac{1}{\alpha}-1}=
     \left\{\begin{array}{cc}
              \infty, & \alpha<1, \\
              |x|, & \alpha=1, \\
              0, & \alpha>1.
            \end{array}\right.
  \end{multline*}
This shows that in the case $\alpha<1$ the series (\ref{eq:GN0}) diverges for all $x$, in the case $\alpha>1$ the series converges for all $x$, and in the case $\alpha=1$ the series (\ref{eq:GN0}) converges if $|x|<1$.

Now we consider the case $\alpha<1$. In this case at $N\to\infty$ the series (\ref{eq:GN0}) diverges. However, it follows from the expression (\ref{eq:RGN0}) that  for some fixed $N$
\begin{equation*}
  \mathcal{R}_N^0(x,\alpha,\theta)=O\left(x^{N+1}\right),\quad x\to0.
\end{equation*}
Consequently, for each $N$ we have
\begin{equation*}
  G(x,\alpha,\theta)=\frac{1-\theta}{2}+\frac{1}{\pi\alpha}\sum_{n=0}^{N-1}\frac{x^{n+1}}{(n+1)!} \Gamma\left(\frac{n+1}{\alpha}\right)\sin\left(\frac{\pi}{2}(n+1)(1-\theta)\right)+O\left(x^{N+1}\right),\quad x\to0.
\end{equation*}
Thus, we have obtained the definition of an asymptotic series. Consequently,
\begin{equation*}
  G(x,\alpha,\theta)\sim\frac{1-\theta}{2}+\frac{1}{\pi\alpha}\sum_{n=0}^{N-1}\frac{x^{n+1}}{(n+1)!} \Gamma\left(\frac{n+1}{\alpha}\right)\sin\left(\frac{\pi}{2}(n+1)(1-\theta)\right),\quad x\to0,\quad\alpha<1.
\end{equation*}

Now we consider the case $\alpha>1$. In this case the series (\ref{eq:GN0}) is convergent. It follows from the expressions (\ref{eq:Gx0}) and (\ref{eq:RGN0}) that
  \begin{equation}\label{eq:absG_GN0}
    \left|G(x,\alpha,\theta)-(1-\theta)/2-G_N^0(x,\alpha,\theta)\right|\leqslant\frac{|x|^{N+1}}{\alpha\pi(N+1)!} \Gamma\left(\frac{N+1}{\alpha}\right).
  \end{equation}
We will find the limit at $N\to\infty$ of the right-hand side of this inequality. Using Stirling’s formula (\ref{eq:Stirling}) and taking into account that $N+2\approx N+1$ at $N\to\infty$, we obtain
\begin{multline*}
\frac{1}{\alpha\pi}\lim_{N\to\infty}\frac{\Gamma\left(\frac{N+1}{\alpha}\right)}{(N+1)!}|x|^{N+1}=
\frac{1}{\alpha\pi}\lim_{N\to\infty}\frac{\Gamma\left(\frac{N+1}{\alpha}\right)}{\Gamma(N+2)}|x|^{N+1}=
    \frac{1}{\alpha\pi}\lim_{N\to\infty}\frac{e^{-\frac{N+1}{\alpha}} \left(\frac{N+1}{\alpha}\right)^{\frac{N+1}{\alpha}-\frac{1}{2}}\sqrt{2\pi}}{e^{-(N+2)}(N+2)^{N+2-1/2}\sqrt{2\pi}} |x|^{N+1}=\\
    =\frac{1}{\alpha\pi}\lim_{N\to\infty}\frac{e^{-\frac{N+1}{\alpha}} \left(\frac{N+1}{\alpha}\right)^{\frac{N+1}{\alpha}-\frac{1}{2}}}{e^{-(N+1)}(N+1)^{N+1-1/2}} |x|^{N+1}=
    \frac{\alpha^{-\frac{1}{2}-\frac{1}{\alpha}}}{\pi}\lim_{N\to\infty}\alpha^{-\frac{N}{\alpha}} e^{(N+1)\left(1-\frac{1}{\alpha}\right)}(N+1)^{(N+1)\left(\frac{1}{\alpha}-1\right)}|x|^{N+1}=\\
    \frac{\alpha^{-\frac{1}{2}-\frac{1}{\alpha}}}{\pi}\lim_{N\to\infty} \exp\left\{(N+1)\left(1-\frac{1}{\alpha}\right)(1-\ln(N+1))-\frac{N}{\alpha}\ln\alpha\right\}|x|^{N+1}
    = \left\{\begin{array}{cc}
               \infty, & \alpha<1, \\
               0, & \alpha>1, \\
               0, & \alpha=1,\ |x|<1, \\
               \infty, & \alpha=1,\ |x|\geqslant1.
             \end{array}\right.
  \end{multline*}
Thus, in the two cases $\alpha>1$ and $\alpha=1,\ |x|<1$ the right side of the inequality (\ref{eq:absG_GN0}) is an element of an infinitesimal sequence. In its turn, this means that  in the above two cases, for any fixed $x$, the sequences $(1-\theta)/2+G_N^0(x,\alpha,\theta)$ converge to the distribution function $G(x,\alpha,\theta)$.  Therefore, in the considered case $\alpha>1$ for any fixed $x$ the distribution function can be represented as an infinite series.
  \begin{equation*}
    G(x,\alpha,\theta)=\frac{1-\theta}{2}+\frac{1}{\pi\alpha}\sum_{n=0}^{\infty}\frac{x^{n+1}}{(n+1)!} \Gamma\left(\frac{n+1}{\alpha}\right)\sin\left(\frac{\pi}{2}(n+1)(1-\theta)\right).
  \end{equation*}

Now we consider the case $\alpha=1$. As shown above, in this case, when the condition $|x|<1$ is met, the right side (\ref{eq:absG_GN0}) is an element of an infinitesimal series. Therefore, for any fixed $|x|<1$ the representation in the form of an infinite series is true for the distribution function
  \begin{equation*}
    G(x,1,\theta)=\frac{1-\theta}{2}+\frac{1}{\pi\alpha}\sum_{n=0}^{\infty}\frac{x^{n+1}}{(n+1)} \sin\left(\frac{\pi}{2}(n+1)(1-\theta)\right).
  \end{equation*}
Thus, the corollary has been proved completely.
  \begin{flushright}
    $\Box$
  \end{flushright}
\end{proof}

As in the case of the probability density, the proved property shows that in the case $\alpha=1$ and $|x|<1$ the series (\ref{eq:Gx0_a1}) converges to the distribution function $G(x,1,\theta)$. It is possible to show that this series converges to the distribution function (\ref{eq:Gx_a1}). We will formulate this result as a remark

\begin{remark}\label{rm:cdf_a1}
In the case $\alpha=1$ for any $-1<\theta<1$ in the region $-1<x<1$ the series (\ref{eq:Gx0_a1}) converges to the distribution function (\ref{eq:Gx_a1}).
\end{remark}

\begin{proof}
To prove this remark, we proceed in the same way as in the proof of remark~\ref{rm:pdf_a1}. Let us show that the expansion of the distribution function (\ref{eq:Gx_a1}) into a Taylor series in the vicinity of the point $x=0$ has the form (\ref{eq:Gx0_a1}). We will use the reduction formulas $\cos\left(\tfrac{\pi}{2}\theta\right)=\sin\left(\tfrac{\pi}{2}(1-\theta)\right)$ and $\sin\left(\tfrac{\pi}{2}\theta\right)=\cos\left(\tfrac{\pi}{2}(1-\theta)\right)$  and will write the distribution function (\ref{eq:Gx_a1}) in the form
\begin{equation*}
  G(x,1,\theta)=\frac{1}{2}+\frac{1}{\pi}\arctan\left(\frac{x-\cos\left(\frac{\pi}{2}(1-\theta)\right)} {\sin\left(\frac{\pi}{2}(1-\theta)\right)}\right).
\end{equation*}

Note that the function $\arctan(x)$ is infinitely differentiable, therefore, expanding it into an infinite series, we obtain
\begin{equation}\label{eq:G1_expan_tmp0}
  G(x,1,\theta)=G(0,1,\theta)+\sum_{n=1}^{\infty}\frac{1}{n!}\left.\frac{d^n G(x,1,\theta)}{dx^n}\right|_{x=0}x^n.
\end{equation}
For the derivative of the order $n$ we have
\begin{multline}\label{eq:G_der}
   \frac{d^n G(x,1,\theta)}{dx^n}=\frac{d^{n-1}}{dx^{n-1}} \frac{dG(x,1,\theta)}{dx}=\frac{1}{\pi}\frac{d^{n-1}}{dx^{n-1}}\frac{1}{1+\left(\frac{x-\cos\left(\frac{\pi}{2}(1-\theta)\right)} {\sin\left(\frac{\pi}{2}(1-\theta)\right)}\right)^2}\frac{1}{\sin\left(\frac{\pi}{2}(1-\theta)\right)}=\\
   =\frac{1}{\pi}\frac{d^{n-1}}{dx^{n-1}}\frac{\sin\left(\frac{\pi}{2}(1-\theta)\right)}{x^2-2x\cos\left(\frac{\pi}{2}(1-\theta)\right)+1}=
   \frac{d^{n-1} g(x,1,\theta)}{dx^{n-1}}.
\end{multline}

Thus, the problem has been reduced to calculating  the derivative $n-1$ of the probability density $g(x,1,\theta)$. However, this problem has been solved by us when proving remark~\ref{rm:pdf_a1}.  Using the formula (\ref{eq:g_n_der}) we get
\begin{equation*}
   \frac{d^{n-1} g(x,1,\theta)}{dx^{n-1}}= \frac{\sin\left(\frac{\pi}{2}(1-\theta)\right)}{\pi} \sum_{k=0}^{\left[\tfrac{n-1}{2}\right]} \frac{(-1)^{n-k-1} (n-1)! (n-k-1)!}{k!(n-2k-1)!}\frac{\left(2x-2\cos\left(\tfrac{\pi}{2}(1-\theta)\right)\right)^{n-2k-1}} {\left(x^2-2x\cos\left(\tfrac{\pi}{2}(1-\theta)\right)+1\right)^{n-k}}
\end{equation*}
Substituting this expression in (\ref{eq:G_der}) and calculating the value of the obtained derivative in the point $x=0$ and then using (\ref{eq:sin_na}), we obtain
\begin{multline*}
  \left.\frac{d^n G(x,1,\theta)}{dx^n}\right|_{x=0}=
  \frac{\sin\left(\frac{\pi}{2}(1-\theta)\right)}{\pi} \sum_{k=0}^{\left[\tfrac{n-1}{2}\right]} \frac{(-1)^{k} (n-1)! (n-k-1)!}{k!(n-2k-1)!}\left(2\cos\left(\tfrac{\pi}{2}(1-\theta)\right)\right)^{n-2k-1}=\\
  = \tfrac{1}{\pi} (n-1)!\sin\left(\tfrac{\pi}{2}n(1-\theta)\right),
\end{multline*}
where it was taken into account that $(-1)^{2n-2-3k}=(-1)^{k}$.

Substituting now the obtained expression for the $n$-th derivative in (\ref{eq:G1_expan_tmp0}) and taking into consideration (\ref{eq:G_x=0}), we get
\begin{equation*}
  G(x,1,\theta)=\tfrac{1}{2}(1-\theta)+\frac{1}{\pi}\sum_{n=1}^{\infty}\frac{x^n}{n}\sin\left(\tfrac{\pi}{2}n(1-\theta)\right)
  =\tfrac{1}{2}(1-\theta)+\frac{1}{\pi}\sum_{k=0}^{\infty}\frac{x^{k+1}}{k+1}\sin\left(\tfrac{\pi}{2}(k+1)(1-\theta)\right).
\end{equation*}
Here, in the last equality, the summation index $n=k+1$ was changed. Thus, the expansion of the distribution function (\ref{eq:Gx_a1}) into a Taylor series in the vicinity of the point $x=0$ exactly coincides with the series (\ref{eq:Gx0_a1}). This completely proves the remark.

\begin{flushright}
  $\Box$
\end{flushright}
\end{proof}

Theorem~\ref{th:Gx0_expan} gives an opportunity to find the range of values of the coordinate $x$ within which the absolute error of calculating $G(x,\alpha,\theta)$ using the expansion (\ref{eq:Gx0}) will not exceed the pre-specified value. Indeed, from (\ref{eq:Gx0}) and (\ref{eq:RGN0}) we have
\begin{equation*}
  \left|G(x,\alpha,\theta)-\tfrac{1}{2}(1-\theta)-G_N^0(x,\alpha,\theta)\right|\leqslant\frac{|x|^{N+1}}{\alpha\pi(N+1)!}
  \Gamma\left(\frac{N+1}{\alpha}\right).
\end{equation*}
If now, for a specified fixed $N$ we set the absolute magnitude of the error $\left|G(x,\alpha,\theta)-\tfrac{1}{2}(1-\theta)-G_N^0(x,\alpha,\theta)\right|=\varepsilon$, then it is possible to introduce the threshold coordinate
\begin{equation}\label{eq:x_eps_cdf}
  x_\varepsilon^N=\left(\frac{\pi\varepsilon\alpha (N+1)!}{\Gamma\left(\frac{N+1}{\alpha}\right)}\right)^{\frac{1}{N+1}}
\end{equation}
This value shows that for all $x$ satisfying the condition $|x|\leqslant x_\varepsilon^N$, the absolute magnitude of the error in calculating the distribution function using the expansion  (\ref{eq:Gx0}) will not exceed the value $\varepsilon$:
\begin{equation}\label{eq_x_eps_cond}
  \left|G(x,\alpha,\theta)-\tfrac{1}{2}(1-\theta)-G_N^0(x,\alpha,\theta)\right|\leqslant\varepsilon,\quad -x_\varepsilon^N\leqslant x\leqslant x_\varepsilon^N.
\end{equation}

\begin{figure}
  \centering
  \includegraphics[width=0.45\textwidth]{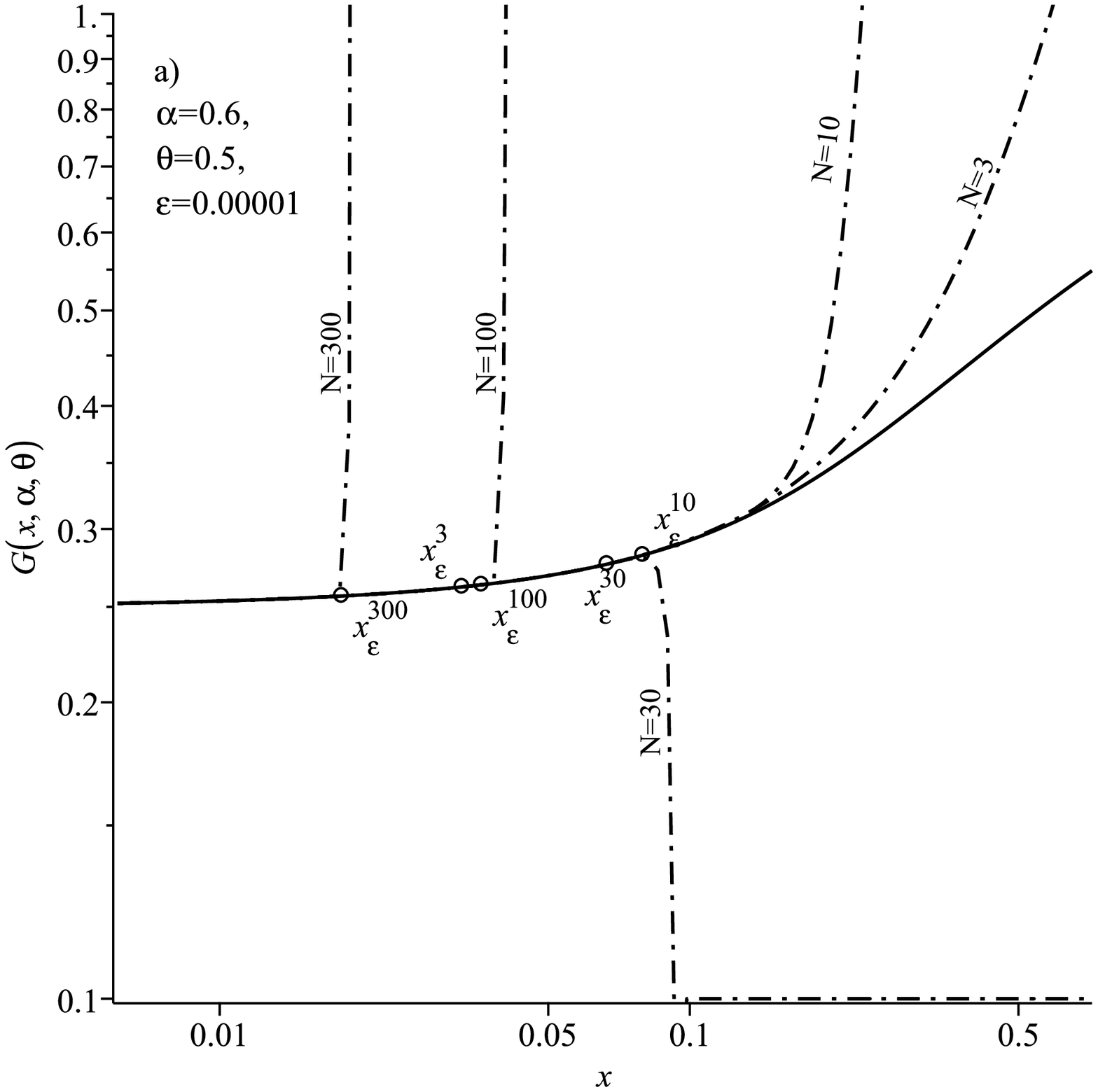}\hfill
  \includegraphics[width=0.45\textwidth]{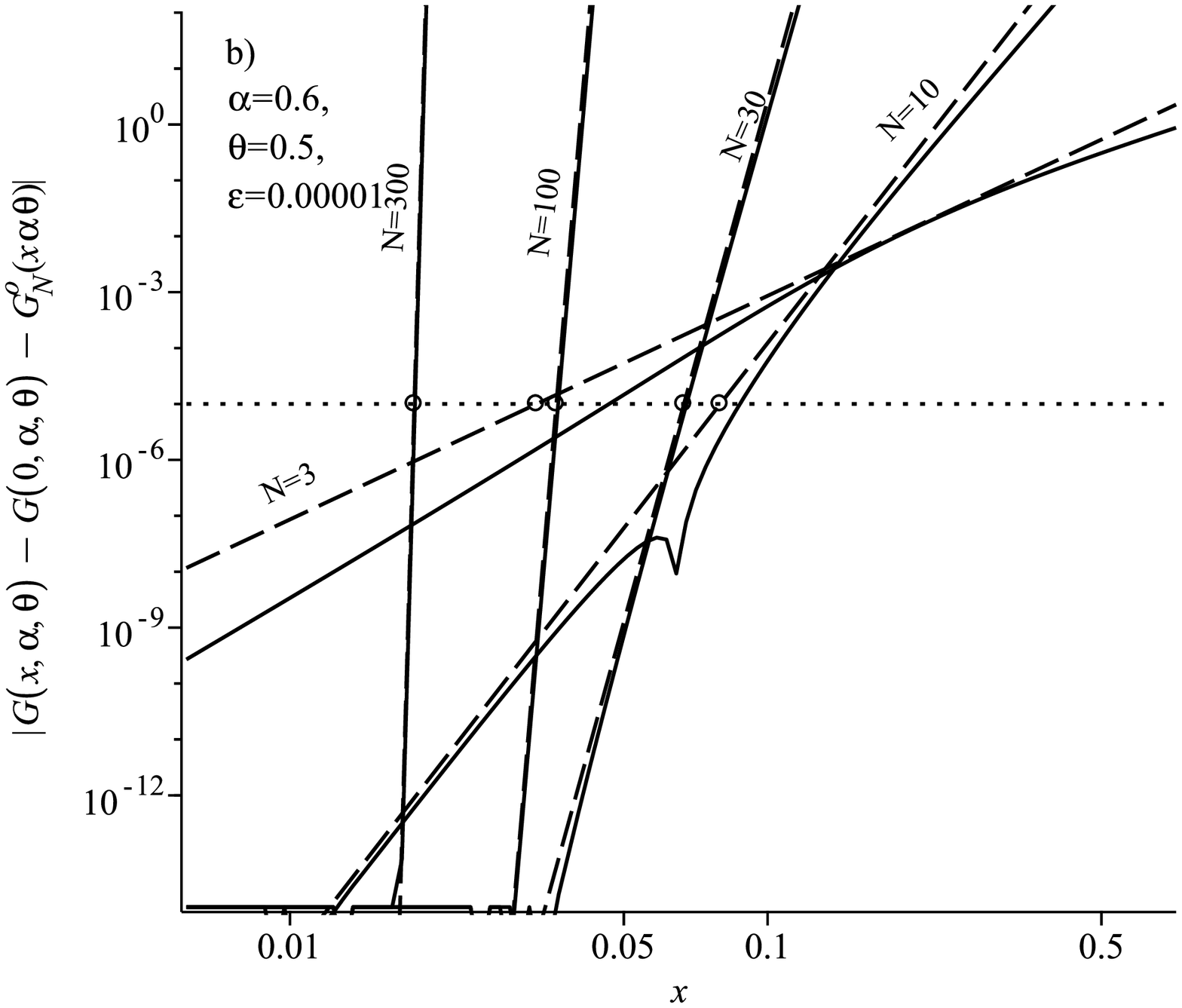}\\
  \caption{(a) Distribution function $G(x,\alpha,\theta)$ for the parameter values shown in the figure.  A solid curve is the integral representation (\ref{eq:cdfI}),  dashed curves are the expansion (\ref{eq:Gx0}) for different values of the number of  $N$ summands in total. Circles are the position of the threshold coordinate (\ref{eq:x_eps_cdf}) for each value $N$. (b) The graph of the absolute error of calculating the distribution function using the expansion (\ref{eq:Gx0}) for the case given in Figure (а).  Solid curves - the exact value of the absolute error $|G(x,\alpha,\theta)- \frac{1}{2}(1-\theta)- G_N^0(x,\alpha,\theta))|$, dashed curves are the estimate (\ref{eq:RGN0}), dotted line shows the position  the specified accuracy level $\varepsilon$. Circles - demonstrate the location of the threshold coordinate (\ref{eq:x_eps_cdf}) for the specified values of $N$}\label{fig:cdfa06}
\end{figure}

\begin{figure}
  \centering
  \includegraphics[width=0.45\textwidth]{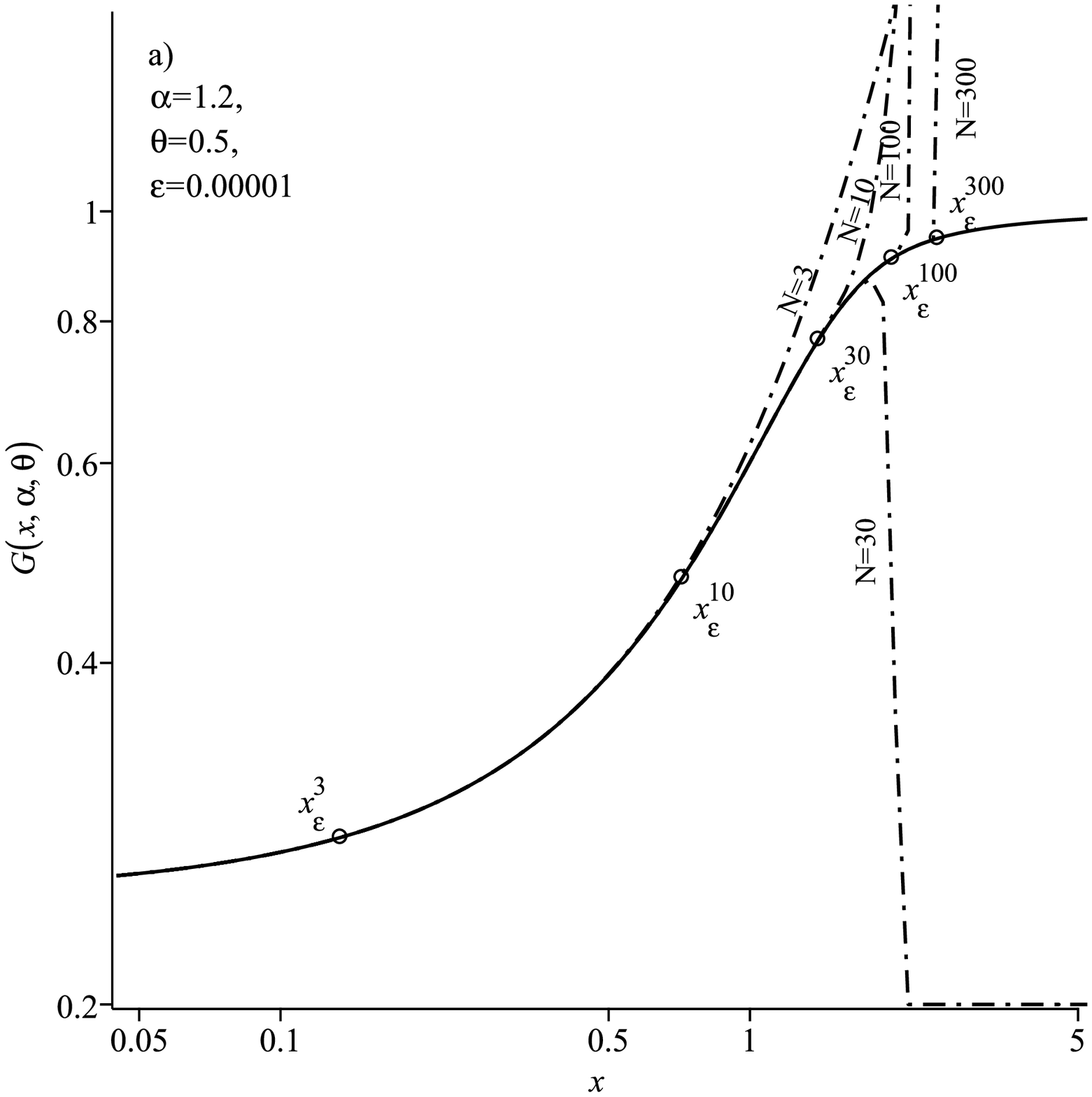}\hfill
  \includegraphics[width=0.45\textwidth]{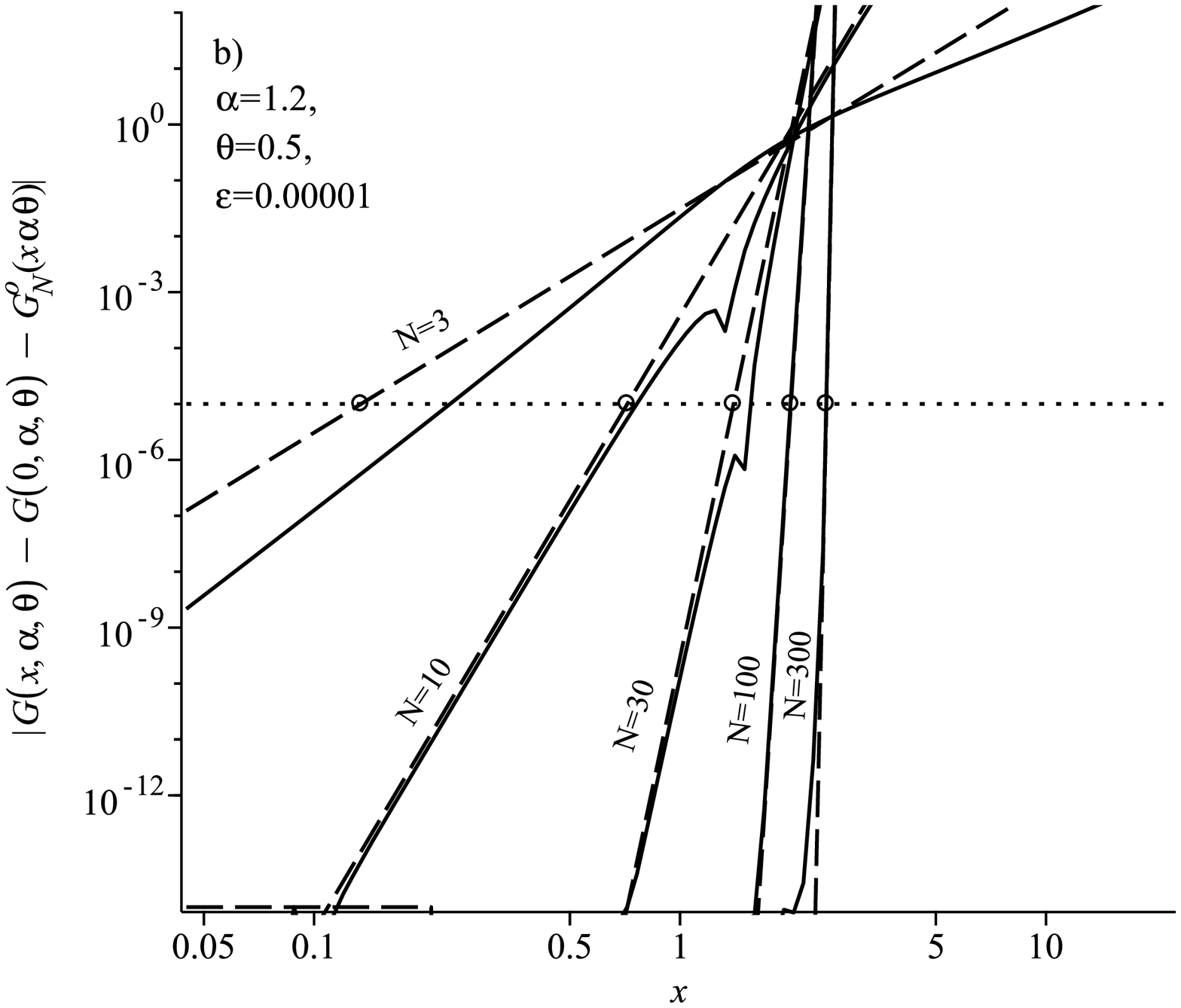}\\
  \caption{(a) Distribution function $G(x,\alpha,\theta)$ for the parameter values shown in the figure.  A solid curve is the integral representation (\ref{eq:cdfI}),  dashed curves are the expansion (\ref{eq:Gx0}) for different values of the number of  $N$ summands in total. Circles are the position of the threshold coordinate (\ref{eq:x_eps_cdf}) for each value of $N$. (b) The graph of the absolute error of calculating the distribution function using the expansion  (\ref{eq:Gx0}) for the case given in Figure (а).  Solid curves - the exact value of the absolute error $|G(x,\alpha,\theta)- \frac{1}{2}(1-\theta)- G_N^0(x,\alpha,\theta))|$, dashed curves are the estimate (\ref{eq:RGN0}), dotted line shows the position  the specified accuracy level $\varepsilon$. Circles - demonstrate the location of the threshold coordinate (\ref{eq:x_eps_cdf}) for the specified values of $N$}\label{fig:cdfa12}
\end{figure}

Figures~\ref{fig:cdfa06}a and \ref{fig:cdfa12}a show the calculation results of the distribution function using the integral representation (\ref{eq:cdfI}) (solid curves)  and using the expansion (\ref{eq:Gx0}) (dash-dotted curves) for parameter values $\alpha=0.6$, $\theta=0.5$ and $\alpha=1.2$, $\theta=0.5$ respectively. These figures contain the results of calculating the distribution function using the expansion (\ref{eq:Gx0}) for the values $N=3,10,30,100,300$. Figures~\ref{fig:cdfa06}b and \ref{fig:cdfa12}b show the results of calculating the absolute error. In these figures the dashed curve corresponds to the estimate of the remainder (\ref{eq:RGN0}), solid curves  -- the exact value of the absolute error $|G(x,\alpha,\theta)- \frac{1}{2}(1-\theta)- G_N^0(x,\alpha,\theta))|$ for the values $N=3,10,30,100,300$.  Here $G(x,\alpha,\theta)$ is the exact value of the distribution function calculated using the representation (\ref{eq:cdfI}), $G_N^0(x,\alpha,\theta)$ is determined by (\ref{eq:GN0}).

From Figures~\ref{fig:cdfa06}b and \ref{fig:cdfa12}b it is clear that the condition  (\ref{eq_x_eps_cond}) is met for all given values of $N$. In these figures the location of the threshold coordinate $x_\varepsilon^N$ is marked with circles and the dotted line corresponds to the specified level of accuracy $\varepsilon$. We can see from the presented figures that for all values of $x$, satisfying the condition $|x|\leqslant x_\varepsilon^N$, both the estimate of the remainder (\ref{eq:RGN0}) (dashed lines), and the exact value of the absolute error (solid curves) are below the specified level of accuracy $\varepsilon$. This confirms the validity of the condition (\ref{eq_x_eps_cond}) and shows that the formula (\ref{eq:x_eps_cdf}) can be used to estimate the boundary value of the coordinate in the expansion (\ref{eq:Gx0}) at which the specified level of accuracy is achieved.

It should be noted that in the case  $\alpha<1$ and $\alpha>1$ the threshold coordinate $x_\varepsilon^N$ behaves differently as the number of summands $N$ in the expansion (\ref{eq:Gx0}) increases. In the case $\alpha<1$ (Fig.~\ref{fig:cdfa06}) an increase in $N$ first increases the threshold coordinate $x_\varepsilon^N$ ($x_\varepsilon^3<x_\varepsilon^{10}$), but with the further increase in  $N$ the threshold coordinate $x_\varepsilon^N$ decreases $(x_\varepsilon^{30}>x_\varepsilon^{100}>x_\varepsilon^{300})$. The threshold coordinate $x_\varepsilon^N$ behaves quite differently in the case $\alpha>1$. In this case with an increase in $N$ the value of the threshold coordinate increases: $x_\varepsilon^{3}<x_\varepsilon^{10}<x_\varepsilon^{30}<x_\varepsilon^{100}<x_\varepsilon^{300}$ (see Fig.~\ref{fig:cdfa12}). Such behavior of the threshold coordinate $x_\varepsilon^N$ is due to the fact that in the case ($\alpha<1$) the series (\ref{eq:GN0}) is divergent, and in the case $\alpha>1$ this series converges (see corollary~\ref{corol:GN0_a><1}).

\begin{figure}
  \centering
  \includegraphics[width=0.45\textwidth]{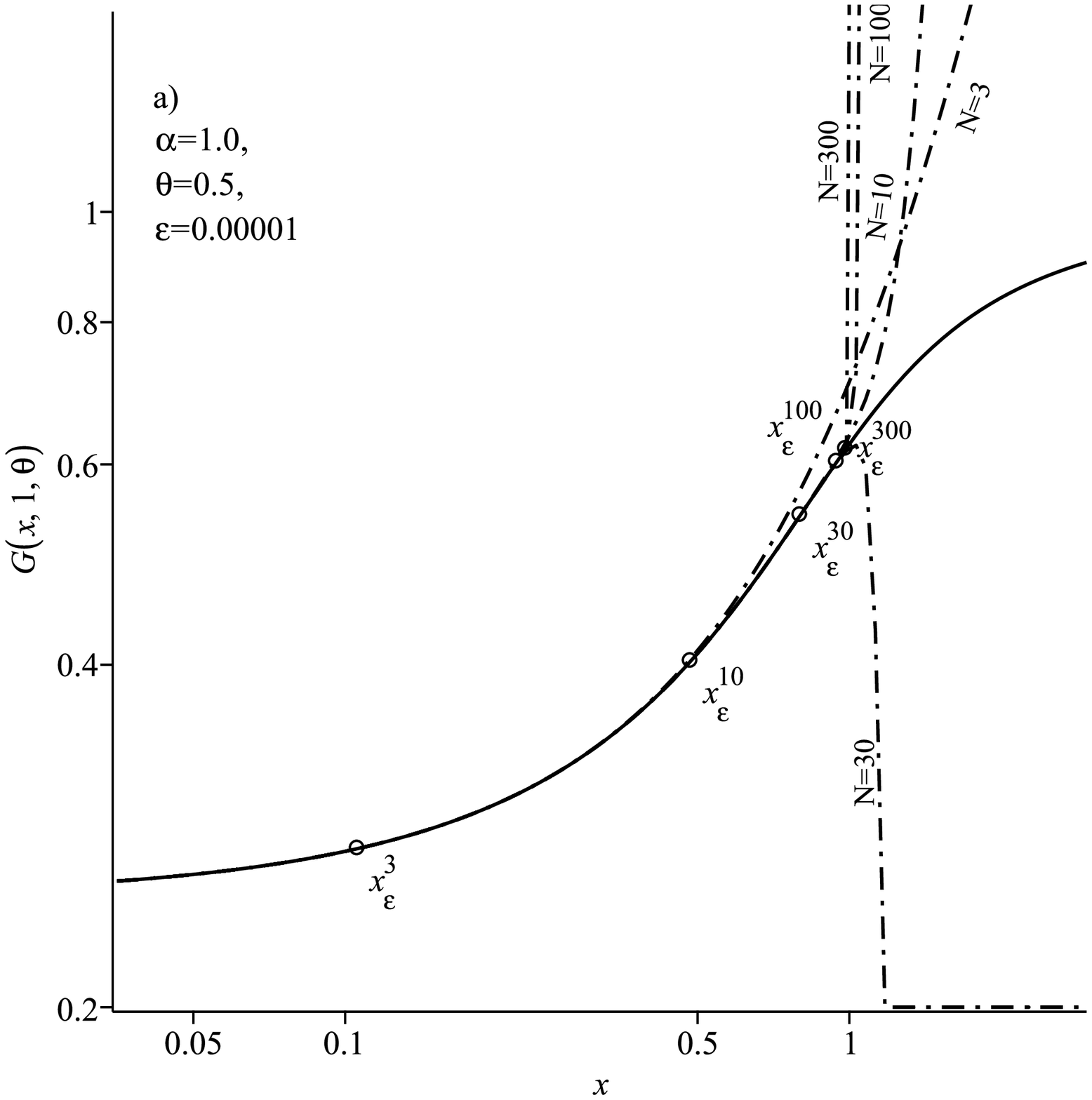}\hfill
  \includegraphics[width=0.45\textwidth]{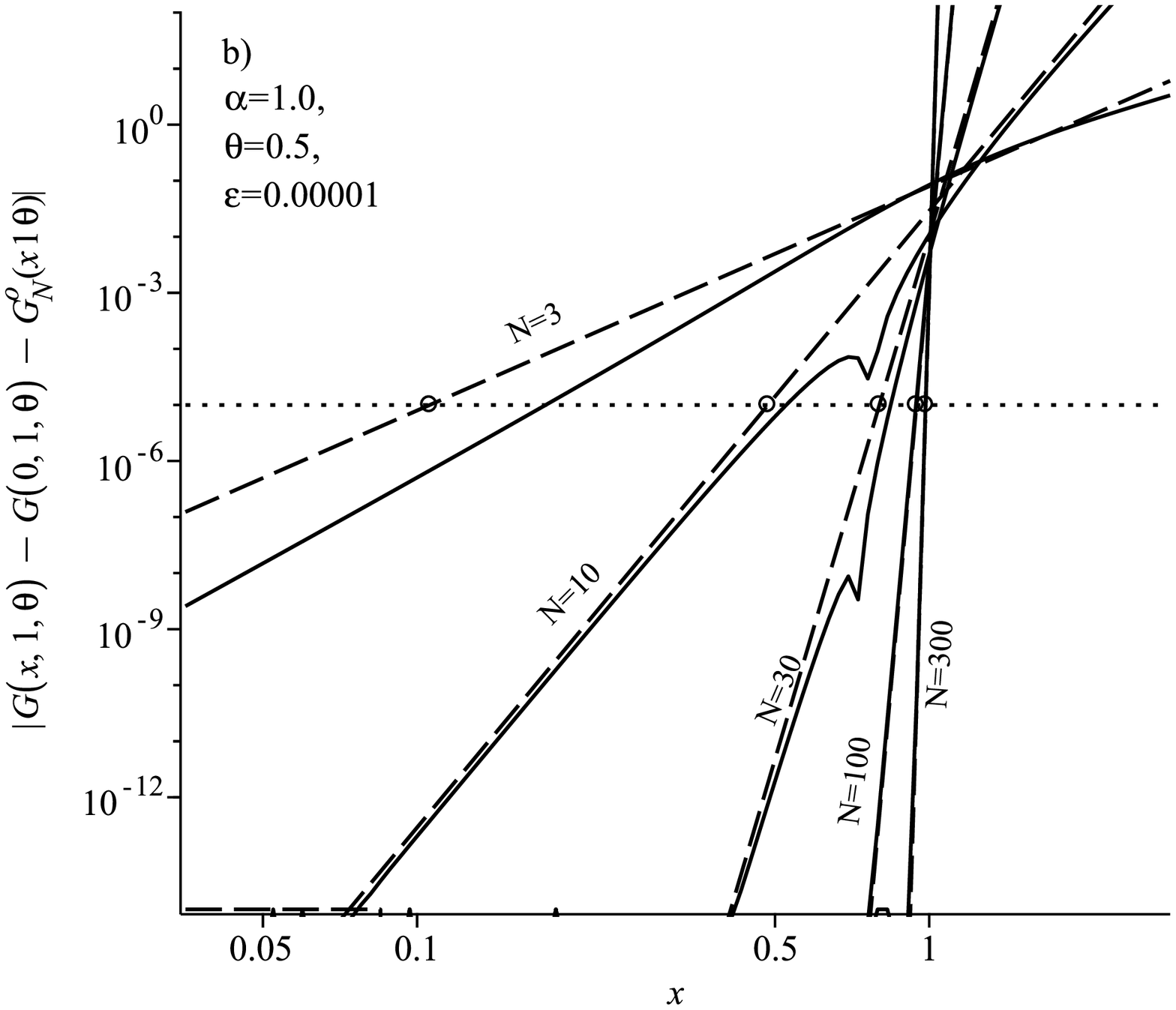}\\
  \caption{(a) Distribution function $G(x,\alpha,\theta)$ for the parameter values shown in the figure.  A solid curve is the formula (\ref{eq:Gx_a1}), dashed curves are the expansion (\ref{eq:Gx0}) for different values of the number of  $N$ summands in total. Circles are the position of the threshold coordinate (\ref{eq:x_eps_cdf}) for each value of $N$. (b) The graph of the absolute error of calculating the distribution function using the expansion (\ref{eq:Gx0}) for the case given in Figure (а).  Solid curves - the exact value of the absolute error $|G(x,\alpha,\theta)- \frac{1}{2}(1-\theta)- G_N^0(x,\alpha,\theta))|$, dashed curves are the estimate (\ref{eq:RGN0}), dotted line – demonstrates the position of the specified accuracy level $\varepsilon$. Circles show the location of the threshold coordinate (\ref{eq:x_eps_cdf}) for the specified values of $N$}\label{fig:cdfa1}
\end{figure}

In the case $\alpha=1$ the threshold coordinate behaves in the same way as the case $\alpha>1$. With an increase in the number of summands of $N$ in the expansion (\ref{eq:Gx0}) the value of the  threshold coordinate $x_\varepsilon^N$ increases. We can see it from Fig.~\ref{fig:cdfa1}, which contains $x_\varepsilon^{3}<x_\varepsilon^{10}<x_\varepsilon^{30}<x_\varepsilon^{100}<x_\varepsilon^{300}$. However, unlike the previous case, $\lim_{N\to\infty} x_\varepsilon^N=1$. Indeed, in the case $\alpha=1$ the formula (\ref{eq:x_eps_cdf}) takes the form $x_\varepsilon^N=(\varepsilon (N+1))^{1/(N+1)}$. Thus,
\begin{equation*}
  \lim_{N\to\infty}x_\varepsilon^N=\lim_{N\to\infty}(\varepsilon (N+1))^{1/(N+1)}= \lim_{N\to\infty}\exp\left\{\frac{\ln(\varepsilon(N+1))}{N+1}\right\}=1.
\end{equation*}

Such behavior of the threshold coordinate $x_\varepsilon^N$ is a consequence of proved corollary~\ref{corol:GN0_a><1}.  Indeed, in the case $\alpha=1$ the series (\ref{eq:GN0}) and, therefore, the representation (\ref{eq:Gx0}) converges in the region $|x|<1$.

The results of calculating the absolute error in the case $\alpha=1$ are given in Fig.~\ref{fig:cdfa1}b. In this figure, the value of the threshold coordinate $x_\varepsilon^N$  for different values $N$ is shown with a circle. We can see from the presented results that for the values  $|x|\leqslant x_\varepsilon^N$ both the estimate of the remainder  (\ref{eq:RGN0}) (dashed lines), and the exact value of the absolute error (solid curves) turn out to be less than specified accuracy level $\varepsilon$ (dotted line). This demonstrates that the use of the formula (\ref{eq:x_eps_cdf}) to estimate the values of the boundary coordinate leads to the validity of the condition (\ref{eq_x_eps_cond}).

\section{ Calculation of the probability density and distribution function for small $x$}\label{sec:calc}

We return to the problem of calculating the probability density of a strictly stable law. As mentioned in the Introduction, the main approach to the calculation of the probability density is to use the integral representation. For a strictly stable law with the characteristic function (\ref{eq:CF_formC}) such an integral representation is determined by the formula (\ref{eq:pdfI}). This formula is valid for any $x\neq0$ and any admissible values of parameters $\alpha$ and $\theta$ except for $\alpha=1$.   However, in practice, it is not possible to calculate the integral in (\ref{eq:pdfI})  numerically for all values of $x$. The reason for this lies in the behavior of the integrand. Figure.~\ref{fig:integrand} shows the graph of the integrand in the formula (\ref{eq:pdfI}) depending on the integration variable  $\varphi$ for different values of $x$. The graph of the function is plotted on a semi-logarithmic scale. We can see from this figure that as the value of $x$ decreases, the integrand turns into a function with a very narrow and sharp peak. With a further decrease in $x$ this peak becomes even narrower and higher. The same behavior of the integrand is also observed for large values of $x$. This leads to the fact that for very small and for very large values of $x$ numerical integration algorithms cannot calculate the integral of this function.

\begin{figure}
  \centering
  \includegraphics[width=0.9\textwidth]{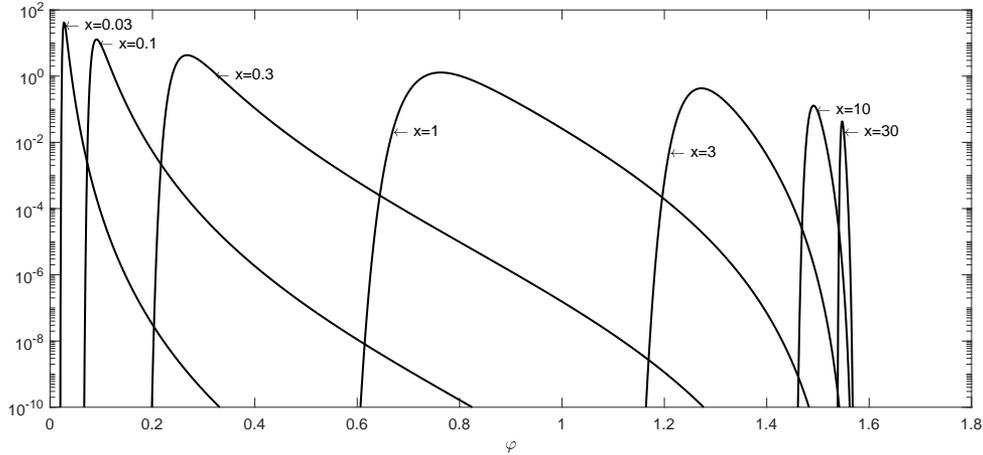}
  \caption{ The integrand of the integral representation (\ref{eq:pdfI}) depending on the variable $\varphi$ for $\alpha=1.1$, $\theta=0$ and specified values of $x$. }\label{fig:integrand}
\end{figure}

Figures~\ref{fig:pdf_a01}~and~\ref{fig:pdf_a12} show the results of calculating the probability density using the integral representation (\ref{eq:pdfI}) (solid curves). Fig.~\ref{fig:pdf_a01} shows the case $\alpha<1$, Fig.~\ref{fig:pdf_a12} shows the case $\alpha>1$.  The Gauss-Kronrod algorithm was used to calculate the integral in the formula (\ref{eq:pdfI}). It is clear from the presented calculations, for small values of $x$ the numerical integration method used is cannot calculate the integral in (\ref{eq:pdfI}). The critical value of the coordinate $x_{\mbox{\scriptsize cr}}$ at which the numerical integration algorithm used begins to produce an incorrect result for $\alpha=0.3$ is $x_{\mbox{\scriptsize cr}}\approx 9\cdot10^{-10}$, for $\alpha=0.6$ $x_{\mbox{\scriptsize cr}}\approx 7\cdot10^{-7}$,  for $\alpha=0.9$ $x_{\mbox{\scriptsize cr}}\approx 3\cdot10^{-4}$ (see~Fig.~\ref{fig:pdf_a01}). In the case $\alpha>1$ (see~Fig.~\ref{fig:pdf_a12}) for the value $\alpha=1.1$ $x_{\mbox{\scriptsize cr}}\approx 3\cdot10^{-6}$, for the value $\alpha=1.4$ $x_{\mbox{\scriptsize cr}}\approx 10^{-10}$, and for the value $\alpha=1.7$ $x_{\mbox{\scriptsize cr}}\approx 10^{-11}$.  Consequently, at $x<x_{\mbox{\scriptsize cr}}$ it is necessary to use other methods to calculate the probability density. The same problem exists for integral representations of the density of stable laws in other parameterizations of the characteristic function (see \cite{Nolan1997,Julian-Moreno2017,Royuela-del-Val2017,Ament2018}). To solve this problem in these works, the authors used various numerical methods, which make it possible to increase the accuracy of the calculation. However, these methods increase the accuracy of calculations, but do not solve the problem completely.

\begin{figure}
  \centering
  \includegraphics[width=0.47\textwidth]{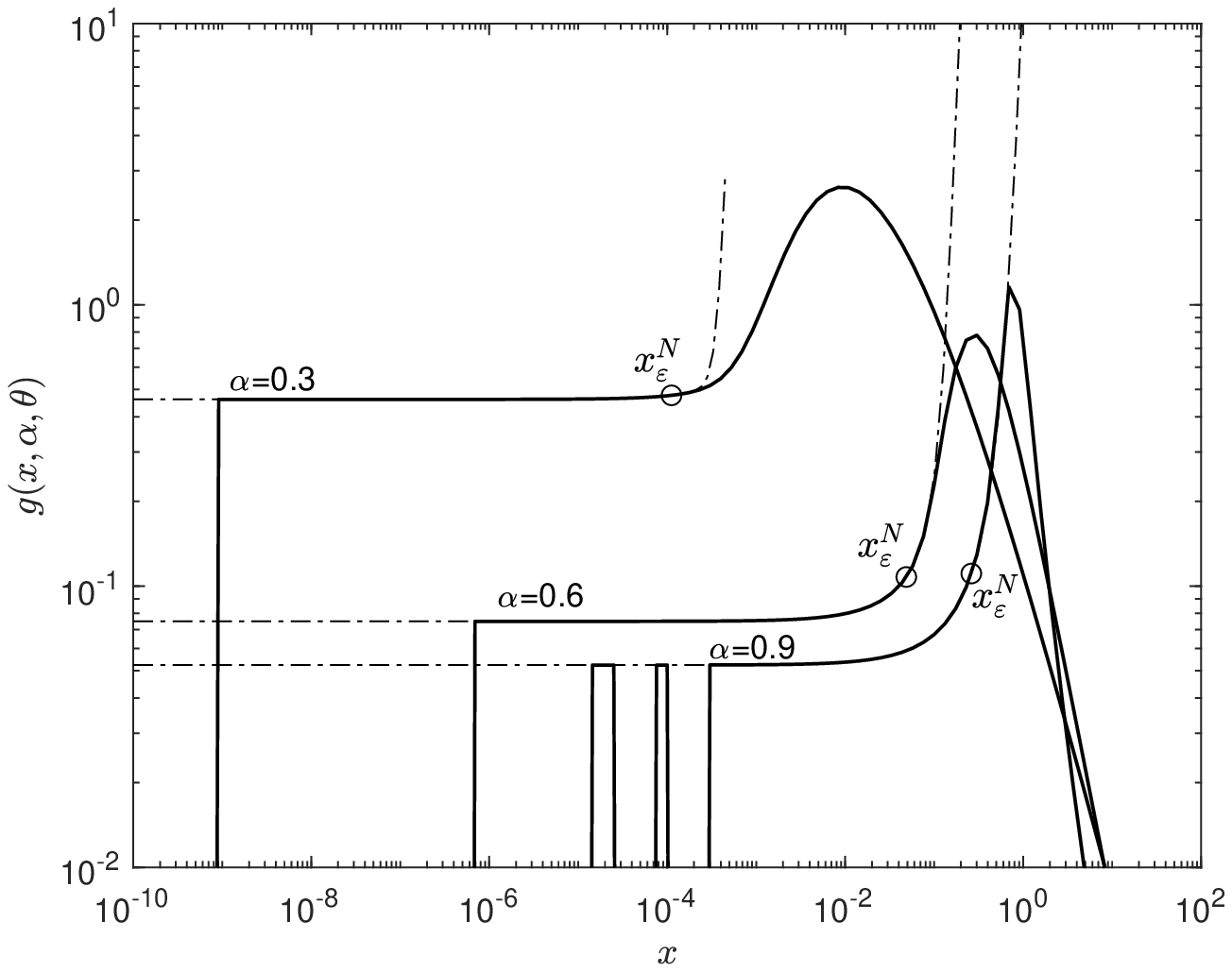}\hfill
  \includegraphics[width=0.47\textwidth]{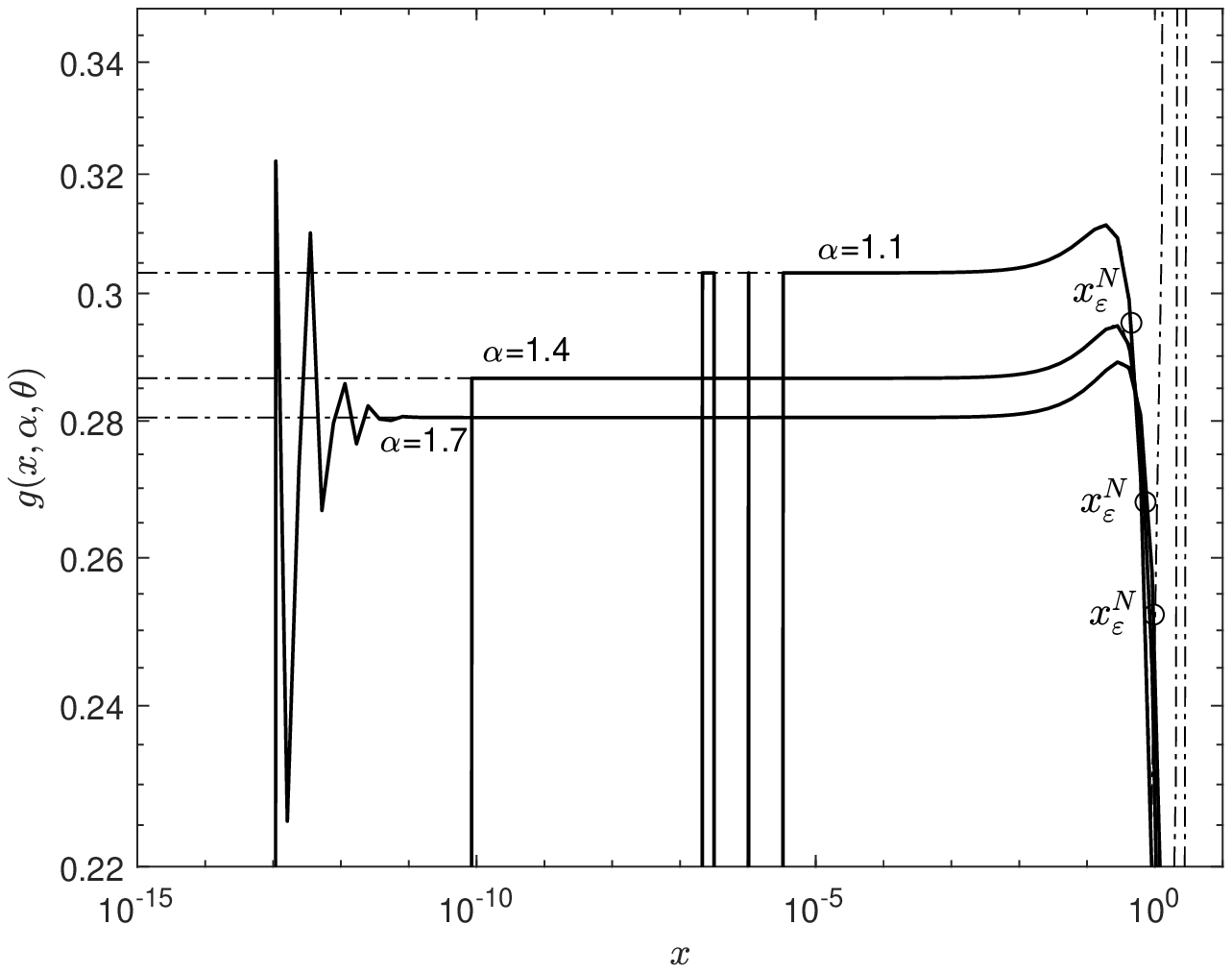}\\
\parbox[t]{0.45\textwidth}{\caption{Probability density in the case $\alpha<1$. Solid curves – the formula (\ref{eq:pdfI}), dash-dotted curves – the representation (\ref{eq:gx0}) for $N=10$, circles – the position of the threshold coordinate $x_\varepsilon^N$ (\ref{eq:x_eps}). While calculating $\varepsilon=10^{-5}$ was used . Calculations are given for the specified values of $\alpha$ and $\theta=0.9$}\label{fig:pdf_a01}}\hfill
\parbox[t]{0.45\textwidth}{\caption{Probability density for the case $\alpha>1$. The notation is the same as in Fig.~\ref{fig:pdf_a01}}\label{fig:pdf_a12}}
\end{figure}

To calculate the density for the values $|x|<x_{\mbox{\scriptsize cr}}$ one should use other representations that do not have any singularities in this area. The most suitable option for this purpose is the power series representation obtained in Theorem~\ref{th:gx0_expan} for the probability density. The estimate of the remainder (\ref{eq:absRN}) obtained in the same theorem  made it possible to obtain the formula  for the threshold coordinate $x_\varepsilon^N$ (\ref{eq:x_eps})  at which the given value of the absolute error $\varepsilon$ is achieved for fixed $N$. This means that in the region $-x_\varepsilon^N\leqslant x\leqslant x_\varepsilon^N$ the absolute error of the density calculation using the series (\ref{eq:gx0}) will not exceed the specified value $\varepsilon$. In Figures~\ref{fig:pdf_a01}~and~\ref{fig:pdf_a12} the dash-dotted curves show the results of calculating the probability density using the series (\ref{eq:gN0}) for the specified values of $\alpha$. The position of the threshold coordinate $x_\varepsilon^N$ is shown with circles. The values $x_\varepsilon^N$ are calculated for the absolute error $\varepsilon=10^{-5}$ and $N=10$. These figures show in the region $x_{\mbox{\scriptsize cr}}\leqslant x\leqslant x_\varepsilon^N$ the results of calculating the probability density using the integral representation (\ref{eq:pdfI}) and using the series (\ref{eq:gN0}) coincide. For the values $|x|\leqslant x_{\mbox{\scriptsize cr}}$ the numerical integration algorithm no longer allows obtaining the correct density value. At the same time, the calculation of the probability density using the series (\ref{eq:gN0}) does not cause any difficulties.  It follows that for the values $|x|\leqslant x_\varepsilon^N$ it is expedient to use the series (\ref{eq:gN0})  to calculate the probability density.        Thus, using theorem~\ref{th:gx0_expan} and, in particular, the series(\ref{eq:gN0}) completely solves the problem of calculating the probability density at $x\to0$.

Similar problems arise when calculating the distribution function using the integral representation (\ref{eq:cdfI}). The integrand in this integral representation also has some singularities at $x\to0$. In the general case, the integrand in (\ref{eq:cdfI}) (see~also~(\ref{eq:cdfI_G+})) behaves in the following way.  In the point of the lower limit $\varphi=-\pi\theta/2$ the integrand is equal to 1, in the point of the upper limit $\varphi=\pi/2$ the value of the integrand is equal to 0. As the variable $\varphi$ increases from the value $-\pi\theta/2$ up to the value $\pi/2$ the integrand decreases monotonically from 1 to 0. However, for very small values $x$ the integrand in (\ref{eq:cdfI_G+}) decreases very sharply from 1 to 0 in a very narrow range $\varphi$. As a result, some numerical integration algorithms cannot recognize such a sharp decrease in the function and give an incorrect integration result. To exclude the possibility of incorrect results completely for small values $x$, it is expedient to use the expansion of the distribution function into a series obtained in theorem~\ref{th:Gx0_expan}. The estimate of the remainder obtained in this theorem made it possible to obtain the formula (\ref{eq:x_eps_cdf}) for the threshold coordinate. The value $x_\varepsilon^N$ enables us to determine the range of $x$ at which the inequality (\ref{eq_x_eps_cond}) is satisfied. In other words, in the range of values $|x|\leqslant x_\varepsilon^N$ the absolute error of calculating the distribution function using the expansion (\ref{eq:Gx0}) will not exceed the value $\varepsilon$, where $\varepsilon$ is given by in advance. Therefore, when calculating the distribution function in the range of values $|x|\leqslant x_\varepsilon^N$ t is expedient to use the expansion (\ref{eq:Gx0}), and for $|x|>x_\varepsilon^N$ the integral representation (\ref{eq:cdfI}).

\section{Conclusion}

The major approach to the calculation of the probability density and the distribution function of stable laws is the use of integral representations. Theoretically, these representations are valid for all values of the coordinate $x$. However, it is not possible to calculate the density numerically for all $x$. Problems arise in the domain of very small and very large values of $x$. Therefore, it is expedient to use other methods for numerical calculations.

The paper considers the problem of calculating the probability density and distribution function in the case of $x\to0$ for a strictly stable law with a characteristic function (\ref{eq:CF_formC}). To solve this problem, expansions of the probability density and distribution function in a power series and estimates of the residual terms for each of the expansions were obtained. Estimates of the threshold coordinates  $x_\varepsilon^N$ were obtained for the expansion of the probability density and the distribution function, which are defined by the expressions (\ref{eq:x_eps})~and~(\ref{eq:x_eps_cdf}), respectively. The threshold coordinate allows one to determine the domain of coordinates $|x|\leqslant x_\varepsilon^N$ within which the absolute computational error will not exceed the specified accuracy level $\varepsilon$. The performed calculations showed that the value of the critical coordinate $x_{\mbox{\scriptsize cr}}$, at which the numerical integration algorithm used starts giving an incorrect result, is significantly less than the threshold coordinate $x_\varepsilon^N$ (see~Fig.~\ref{fig:pdf_a01}~and~\ref{fig:pdf_a12}). This fact shows that in the domain $|x|\leqslant x_\varepsilon^N$ it is possible to use theorems~\ref{th:gx0_expan}~and~\ref{th:Gx0_expan} to calculate the probability density and distribution function.

The analysis of the obtained series made it possible to confirm both the known properties of these series and to establish new properties, as well as to improve the known estimates of the remainder terms. It was shown that in the case $\alpha<1$ the power series were divergent for any $x$ at $N\to\infty$. In this case these series are asymptotic at $x\to0$. In the case $\alpha>1$ the obtained series are convergent for all admissible $x$. In this case, representations in the form of infinite series are valid for the probability density and distribution function (see corollaries~\ref{corol:gN0_a><1}~and~\ref{corol:GN0_a><1}). These results are known and were previously obtained for the characteristic function in parameterization <<B>> in the works \cite{Bergstrom1952}, \cite{Feller1971_V2_en} (see Chapter~17, \S7), \cite{Zolotarev1986} (see \S2.4~and~\S2.5),\cite{Uchaikin1999} (see~\S4.2, \S4.3). It can be shown that the expansions from corollaries~\ref{corol:gN0_a><1}~and~\ref{corol:GN0_a><1} in the cases $\alpha<1$ and $\alpha>1$ completely correspond to the expansions in the mentioned works. Examining the case $\alpha=1$ helped us establish that the expansions of the probability density and distribution function converged to the probability density (\ref{eq:gx_a1}) and the distribution function (\ref{eq:Gx_a1})  in the domain $|x|<1$ for any $-1<\theta<1$ (see Remarks~\ref{rm:pdf_a1}~and~\ref{rm:cdf_a1}).

The paper improves the estimates of the remainder terms in the expansions of the probability density and the distribution function defined by the formulas (\ref{eq:absRN}) and (\ref{eq:RGN0}). The estimate of the remainder term obtained earlier (see~\cite{Zolotarev1986}, formula~(2.5.2)) refers to the expansion of the probability density in parameterization  <<B>> and in the case of $\alpha<1$ has the form
\begin{equation}\label{eq:RN_zlt}
  |R_N|\leqslant\frac{1}{\pi\alpha}\frac{\Gamma\left(\frac{N+1}{\alpha}\right)}{N!} \left(\cos\left(\frac{\pi}{2}\alpha\beta\right)\right)^{-(N+1)/\alpha}.
\end{equation}
We will take the relation $\theta=\beta K(\alpha)/\alpha$ into account, where $K(\alpha)=\alpha-1+\sign(1-\alpha)$, which relates the asymmetry parameter $\beta$ in parameterization <<B>>  to the asymmetry parameter $\theta$ in parameterization <<C>>. In the case $\alpha<1$ this relation gives $\beta=\theta$. Now comparing (\ref{eq:RN_zlt}) and (\ref{eq:absRN}) we see that
\begin{equation*}
|R_N^0|\leqslant|R_N|.
\end{equation*}
The sign of equality is achieved here only in the case $\beta=0$.

Finishing this paper the following should be noted. In the previous article \cite{Saenko2020b} it was noted that when calculating the integral in the representation (\ref{eq:pdfI}) numerical integration algorithms have difficulties in the domain of small values of the coordinate $x$, in the domain of large values of the coordinate $x$ and in the domain of values of the characteristic parameter $\alpha\approx1$. The first problem out of these three ones has been solved in this paper. The assumption made in the work \cite{Saenko2020b} that  the cause of the problem is associated with the behavior of the integrand in (\ref{eq:pdfI}), was correct.
Indeed, this integrand at smaller values of $x$ starts acting as a singular function which makes it impossible to use for numerical algorithms to calculate the integral of it. Therefore, to calculate the density in the indicated domain $x$ it is necessary to use series expansions of the density.  The reason for the calculation difficulties in the second case is also the behavior of the integrand in the representation (\ref{eq:pdfI}). As shown in section~\ref{sec:calc} with large $x$ it behaves as a singular function. Therefore, to calculate the density in this domain of the coordinate, it is also expedient to use expansions of the density in a series.  To solve the third problem, one can use the method proposed in the paper \cite{Matsui2006}. In this paper, to calculate the density, the authors propose to use the series expansion of a strictly stable law in view of the parameter $\alpha$. However, the solution of each of the remaining two problems requires further research, which is beyond the scope of this paper.

\AcknowledgementSection

The project has been done under financial support of the Russian Foundation for Basic Research (grants \No  19-44-730005 and 20-07-00655)

\appendix

\section{ Approximation of the gene expression by fractionally stable laws}
It was mentioned in the introduction that the obtained expansions of the probability density of a strictly stable law are useful in problems related to the calculation of the probability density of a fractionally stable law. Indeed, the probability density of the fractionally stable law $q(x,\alpha,\beta,\theta)$ is determined by the Mellin transform of two strictly stable laws
\begin{equation}\label{eq:fsdPDF}
  q(x,\alpha,\beta,\theta)=\int_{0}^{\infty} g(xy^{\beta/\alpha},\alpha,\theta,\lambda)g(y,\beta,1)y^{\beta/\alpha}dy,
\end{equation}
where $g(x,\alpha,\theta)$ and  $g(y,\beta,1)$ -- densities of strictly stable and one-sided strictly stable laws with the characteristic function (\ref{eq:CF_formC}) \cite{Kolokoltsov2001,Bening2006,Saenko2020c}. Here, the characteristic exponents $\alpha$ and $\beta$ vary within $0<\alpha\leqslant2$ and $0<\beta\leqslant1$, the asymmetry parameter $\theta$ takes the values within the interval $|\theta|\leqslant\min(1,2/\alpha-1)$ and $\lambda>0$ is the scaling parameter.

From the formula (\ref{eq:fsdPDF}) it is clear that it is necessary to be able to calculate densities of the strictly stable laws $g(y,\alpha,\theta,\lambda)$ and $g(y,\beta,1)$ for the calculation of density $q(x,\alpha,\beta,\theta,\lambda)$. The integral representation (\ref{eq:pdfI}) is used to calculate these densities.  In this regard, at this stage, certain problems may arise with the calculation of the improper integral in (\ref{eq:fsdPDF}). Indeed, to calculate the integral in (\ref{eq:fsdPDF}) the numerical integration algorithm calculates the integrand at some integration nodes $y_i$. If it turns out that the next integration node $y_i$ is less than the value of the critical coordinate $y_{\mbox{\scriptsize cr}}$, i.е. $|y_i|<y_{\mbox{\scriptsize cr}}$, then the numerical integration algorithm will be unable to calculate the densities $g(y_i,\alpha,\theta)$ and $g(y_i,\beta,1)$. This will lead to an incorrectly calculated density value $q(x,\alpha,\beta,\theta)$. To eliminate this problem and to calculate the densities $g(y_i,\alpha,\theta)$ and $g(y_i,\beta,1)$ in the case $|y_i|<y_{\mbox{\scriptsize cr}}$ it is expedient to use the expansion (\ref{eq:gN0}). Thus, the use of the expansion (\ref{eq:gN0}) will give an opportunity to exclude the integration error associated with the singular behavior of the integrand in the representation (\ref{eq:pdfI}) at small values of $y$.
\begin{figure}
  \centering
  \includegraphics[width=0.47\textwidth]{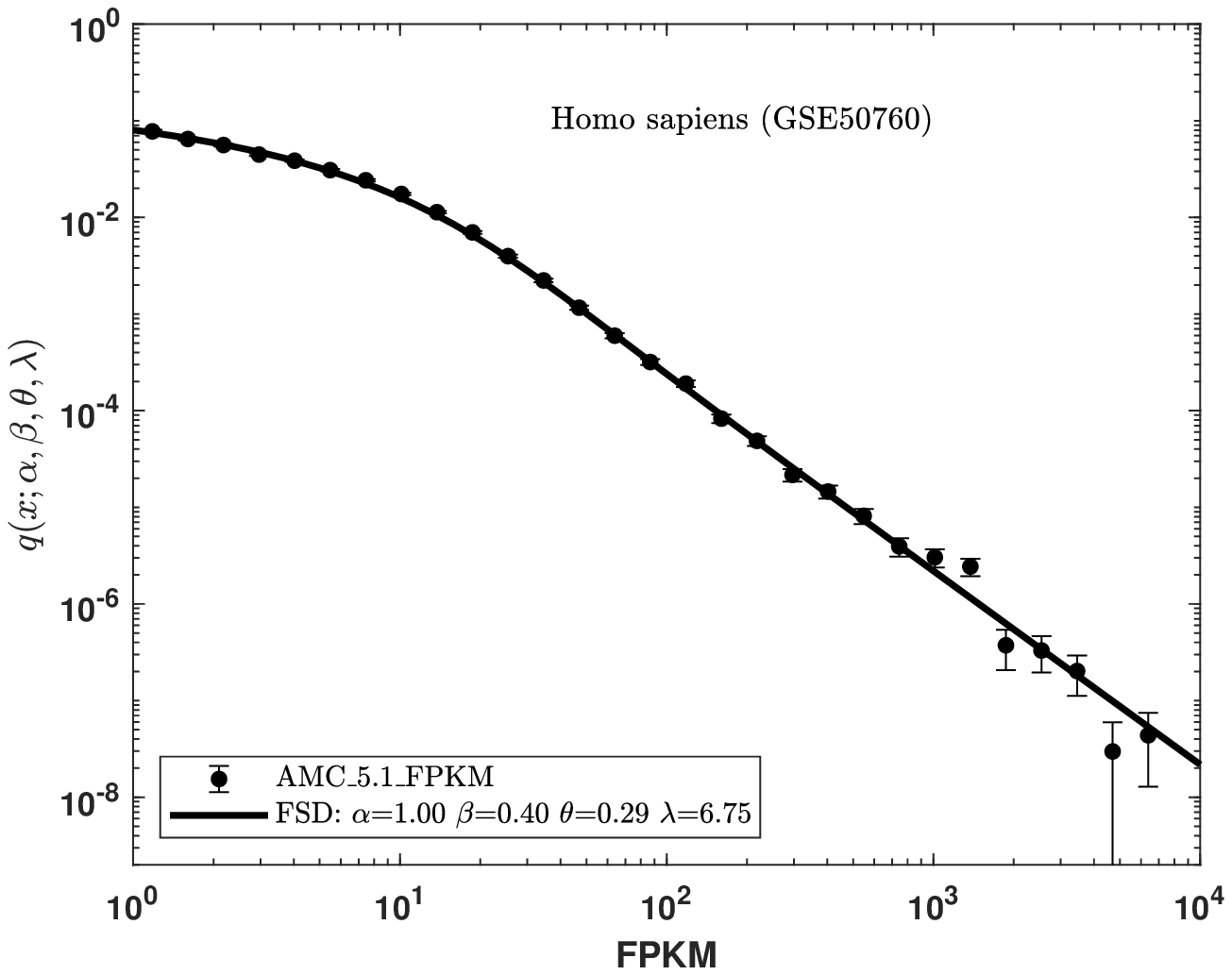}\hfill
  \includegraphics[width=0.47\textwidth]{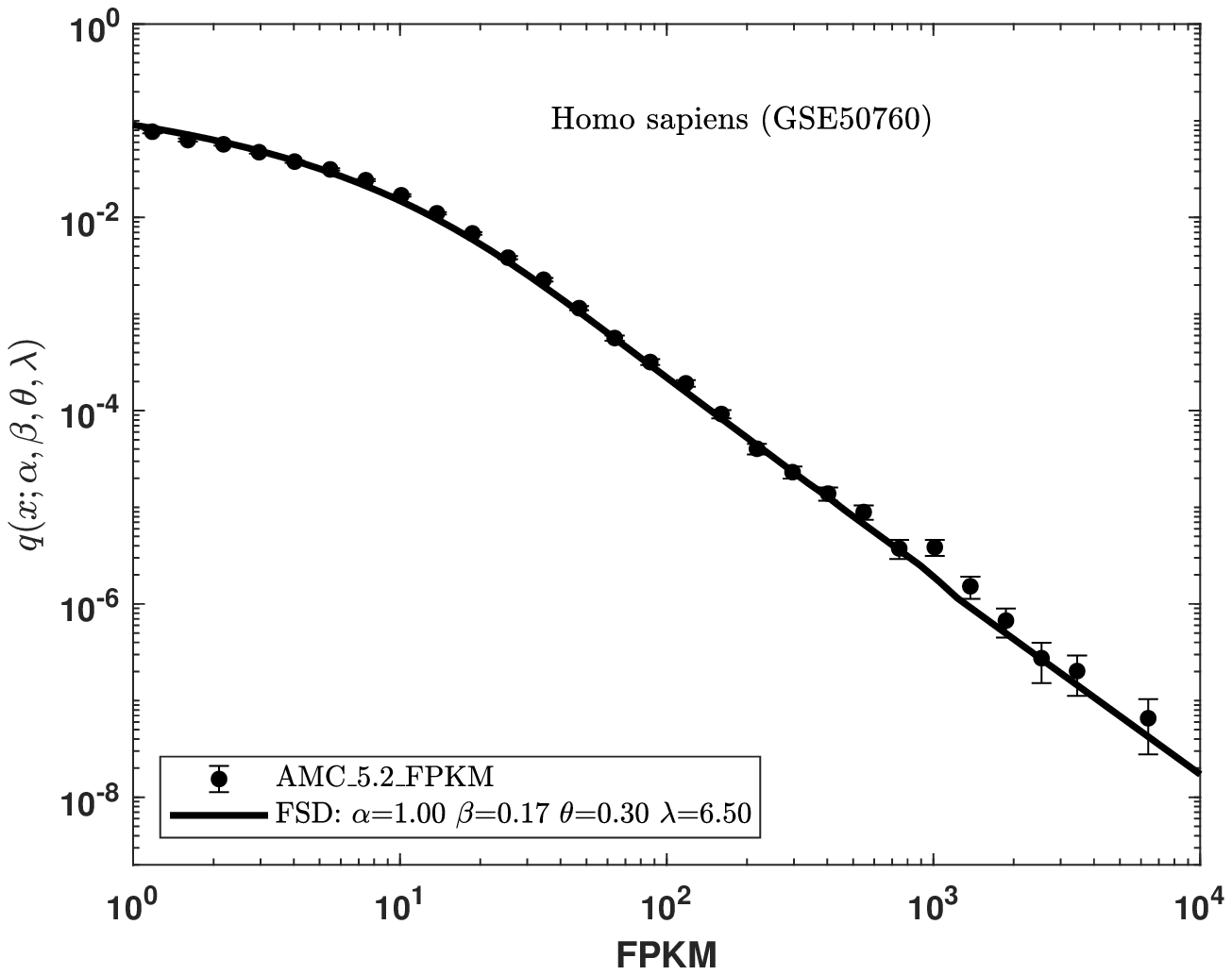}\\
  \caption{Probability density of gene expression obtained with the use of NGS technology. Dots are density histogram of the experimental data on human tissue gene expression. (GSE50760 series), a solid curve is the density (\ref{eq:fsdPDF}) for the specified parameter values}\label{fig:GeneExpress}
\end{figure}

Fractionally stable distributions turn out to be a convenient tool to describe the probability density distribution of gene expression obtained by means of NGS technology. Fig.~\ref{fig:GeneExpress} shows the approximation results of the probability density distribution of gene expression obtained with NGS technology using the density of a fractional stable law.  In these figures, the dots are the histogram of the probability density of gene expression, the solid curve is the density (\ref{eq:fsdPDF}). The density parameters $q(x,\alpha,\beta,\theta,\lambda)$ are shown in the figures. As we can see from these figures the density $q(x,\alpha,\beta,\theta,\lambda)$ approximates the experimental data quite well in a very wide range of values.

It should be noted that the parameters $\alpha,\beta,\theta,\lambda$, given in Fig.~\ref{fig:GeneExpress} were estimated from the experimental data using the minimum distance method, which is based on the distance $\chi^2$ \cite{Saenko2016}. However, this method of parameter estimation is not effective. To build an effective estimate of the parameters of fractionally stable distributions, it is necessary to build an estimate based on the maximum likelihood method. Until now, the creation of such an estimate has met with some difficulties.They are related to the fact that to calculate the likelihood function, it is necessary to calculate the density of the fractionally stable law at the points determined by the original data sample. Taking into account that the density $q(x,\alpha,\beta,\theta,\lambda)$ is calculated according to the formula (\ref{eq:fsdPDF}), then at small values of the coordinate $x$ or or the integration variable $y$ the algorithm numerical integration gave the wrong result. This, in turn, led to an error in estimating the distribution parameters. Thus, the use of the expansion (\ref{eq:gN0}) when calculating the integral in (\ref{eq:fsdPDF}) will give an opportunity to calculate the density correctly, which in turn will allow implementing the methods for estimating the parameters of fractionally stable and strictly stable distributions based on the maximum likelihood method.

\bibliographystyle{elsarticle-num}
\bibliography{d:/bibliography/library}
\end{document}